\documentclass[a4paper,10pt,final]{article}
\usepackage{amsthm,amsfonts,amsmath}
\usepackage{mathrsfs}
\usepackage{color}
\usepackage{hyperref}
\usepackage{a4wide}
\usepackage{enumerate,float}

\usepackage{pgf,tikz}
\usetikzlibrary{arrows}
\usetikzlibrary{patterns}

\newtheorem{proposition}{Proposition}[section]
\newtheorem{theorem}[proposition]{Theorem}
\newtheorem{lemma}[proposition]{Lemma}

\newtheorem{definition}[proposition]{Definition}

\newenvironment{proofof}[1]{\smallskip\noindent{\textbf{Proof~of~#1.}}%
  \hspace{1pt}}{\hspace{-5pt}{\nobreak\quad\nobreak\hfill\nobreak%
    $\square$\vspace{2pt}\par}\smallskip\goodbreak}

\numberwithin{equation}{section}
\allowdisplaybreaks[4]

\renewcommand{\phi}{\varphi}
\renewcommand{\epsilon}{\varepsilon}
\renewcommand{\theta}{\vartheta}
\renewcommand{\L}[1]{\mathbf{L^#1}}
\newcommand{\Lloc}[1]{\mathbf{L^{#1}_{loc}}}
\newcommand{\C}[1]{\mathbf{C^{#1}}}

\newcommand{\W}[2]{\mathbf{W^{#1,#2}}}
\newcommand{\BV}{\mathbf{BV}}
\newcommand{\modulo}[1]{{\left|#1\right|}}
\newcommand{\norma}[1]{{\left\|#1\right\|}}
\newcommand{\reali}{{\mathbb{R}}}

\newcommand{\tv}{\mathop\mathrm{TV}}
\renewcommand{\O}{\mathinner{\mathcal{O}(1)}}

\newcommand{\pint}[1]{\mathaccent23{#1}}

\renewcommand{\d}[1]{\mathinner{\mathrm{d}{#1}}}

\newcommand{\wtv}{\mathop\mathrm{WTV}}

\newcommand{\wsto}{{\mathinner{\stackrel{\star}{\rightharpoonup}}}}

\begin{document}

\title{$\BV$ Solutions to $1$D Isentropic Euler Equations \\ in the
  Zero Mach Number Limit}

\author{Rinaldo M. Colombo$^1$ \and Graziano Guerra$^2$}

\footnotetext[1]{INDAM Unit, University of Brescia,
  Italy. \texttt{rinaldo.colombo@unibs.it}}

\footnotetext[2]{Department of Mathematics and Applications,
  University of Milano - Bicocca,
  Italy. \texttt{graziano.guerra@unimib.it}}

\maketitle

\begin{abstract}
  \noindent
  Two compressible immiscible fluids in 1D and in the isentropic
  approximation are considered. The first fluid is surrounded and in
  contact with the second one. As the Mach number of the first fluid
  vanishes, we prove the rigorous convergence for the fully
  non--linear compressible to incompressible limit of the coupled
  dynamics of the two fluids. A key role is played by a suitably
  refined wave front tracking algorithm, which yields precise $\BV$,
  $\L1$ and weak* convergence estimates, either uniform or explicitly
  dependent on the Mach number.

  \medskip


  \medskip

\end{abstract}

\section{Introduction}
\label{sec:Intro}

This paper is devoted to the compressible to incompressible limit in
the equations of isentropic gas dynamics, a widely studied subject in
the literature, see for instance the well known
results~\cite{KlainermanMajda1981, KlainermanMajda1982,
  MetivierSchochet2001, Schochet1986}, the more
recent~\cite{JiangYong}, the review~\cite{Schochet2005} with the
references therein and the monograph~\cite{Feireisl} for the Navier
Stokes equations. For Euler equations, the usual setting considers
regular solutions, whose existence is proved only for a finite time,
to the compressible equations in $2$ or $3$ space dimensions. As the
Mach number vanishes, these solutions are proved to converge to the
solutions to the incompressible Euler equations.

Consider for instance the isentropic Euler equations in the three
dimensional space:
\begin{displaymath}
  \begin{cases}
    \partial_{t}\rho+\nabla\cdot \left(\rho u\right)=0
    \\
    \partial_{t}\left(\rho u\right)+\nabla \cdot \left(\rho u\otimes
      u\right) + \nabla \overline{P}\left(\rho\right)=0 \,,
  \end{cases}
  \begin{array}{l}
    \overline{P} (\rho) > 0  \,,\quad \overline{P}'\left(\rho\right)>0,
    \\
    (t,x)\in\left[0,+\infty\right[ \times \reali^{3}.
  \end{array}
\end{displaymath}
where $\rho$ is the fluid density, $u$ is its speed and $\overline{P}
(\rho)$ is the pressure. For smooth solutions, this system is
equivalent to
\begin{equation}
  \label{eq:prima}
  \begin{cases}
    \partial_{t}\rho+u\cdot \nabla \rho + \rho \, \nabla\cdot u=0
    \\
    \partial_{t}u + u\cdot \nabla u + \frac{1}{\rho} \, \nabla
    \overline{P}\left(\rho\right)=0 \,.
  \end{cases}
\end{equation}
The Mach number is the ratio between the speed of the particles and
the sound speed; it can be introduced into the equations in at least
two different ways~\cite{Schochet2005}.

First, following~\cite{MetivierSchochet2001}, since the incompressible
limit can be understood as the limit when the Mach number tends to
zero, one begins by rescaling the fluid velocity $u\to\kappa \, u$
where $\kappa$ is a small parameter that eventually converges to
zero. In order to capture the motion of the particles traveling with a
small speed of order of $\kappa$ one needs a space--time rescaling,
$\frac{x}{t}\to \kappa \, \frac{x}{t}$, which allows to obtain, in the
rescaled variables, the system
\begin{equation}
  \label{eq:primaseconda}
  \begin{cases}
    \partial_{t}\rho+u\cdot \nabla \rho + \rho\nabla\cdot u=0\\[10pt]
    \partial_{t}u + u\cdot \nabla u +
    \frac{1}{\rho}\frac{1}{\kappa^{2}}\nabla
    \overline{P}\left(\rho\right)=0 \,.
  \end{cases}
\end{equation}

Alternatively, the same system is considered
in~\cite{KlainermanMajda1982}, but motivated by the following
approach, see~\cite{KlainermanMajda1982,Schochet1986}. Consider fluids
having equations of state $\overline P_{\kappa}(\rho)$, parametrized
by $\kappa$, such that the speed of sound
$\sqrt{\overline{P}_{\kappa}^{\prime}(\rho)}\to +\infty$ as $\kappa\to
0$:
\begin{equation}
  \label{eq:seconda}
  \begin{cases}
    \partial_{t}\rho+u\cdot \nabla \rho + \rho\nabla\cdot u=0
    \\[10pt]
    \partial_{t}u + u\cdot \nabla u + \frac{1}{\rho}\nabla \overline
    P_{\kappa}\left(\rho\right)=0.
  \end{cases}
\end{equation}

The two approaches coincide if the one parameter family of pressure
laws $\overline P_{\kappa}\left(\rho\right)$ satisfies
\begin{equation}
  \label{eq:pfamprop}
  \overline
  P_{\kappa}^{\prime}\left(\rho\right)=\frac{1}{\kappa^{2}}\overline
  P^{\prime}\left(\rho\right),
\end{equation}
where $\overline P$ is the fixed pressure law as
in~\eqref{eq:primaseconda}.

In the incompressible limit, the density is constant in time and space
so that the functional dependence of the pressure on the density is
lost. Therefore, it is convenient to use the pressure instead of the
density as unknown variable. Since $\overline
P_{\kappa}^{\prime}\left(\rho\right)>0$, we can take the inverse
function $R_{\kappa}\left(p\right)= \left(\overline
  P_{\kappa}\right)^{-1}\left(p\right)$ and rewrite~\eqref{eq:seconda}
using the pressure $p$ as unknown:
\begin{displaymath}
  \begin{cases}
    \frac{R_{\kappa}^{\prime}\left(p\right)}{R_{\kappa}\left(p\right)}
    \left[\partial_{t}p+u\cdot \nabla p\right] + \nabla\cdot u=0
    \\[10pt]
    \partial_{t}u + u\cdot \nabla u +
    \frac{1}{R_{\kappa}\left(p\right)}\nabla p=0 \,.
  \end{cases}
\end{displaymath}
As $\kappa\to 0$,
$\overline{P}_{\kappa}^{\prime}\left(\rho\right)\to+\infty$, therefore
$R_{\kappa}^{\prime}\left(p\right)\to 0$, and
$R_{\kappa}\left(p\right)\to \bar \rho$, where $\bar\rho$ is the
constant density at the incompressible limit. Formally, we get the
incompressible equations
\begin{displaymath}
  \begin{cases}
    \nabla\cdot u=0
    \\[7pt]
    \partial_{t}u + u\cdot \nabla u + \frac{1}{\bar \rho}\nabla p=0
    \,.
  \end{cases}
\end{displaymath}
In~\cite{KlainermanMajda1981, KlainermanMajda1982} this limit is
proved to hold for smooth solutions and small times. The heart of the
matter is finding energy estimates, uniform in the small parameter
$\kappa$.

Here, we obtain similar convergence results, in a 1D setting, for
\emph{all times} and within the framework of merely $\BV$ \emph{weak
  entropy solutions}.

\smallskip

The next section describes the physical
setting. Section~\ref{sec:Main} presents the key estimates and the
main convergence results. All technical details are deferred to
Section~\ref{sec:TD}.

\section{Two Immiscible Fluids}

In a 1D setting, an incompressible fluid behaves like a solid and its
speed is constant in space. Therefore, we consider two compressible
immiscible fluids and let one of the two become incompressible,
yielding a singular limit for a free boundary problem. Below, we
consider a volume of a compressible inviscid fluid, say the
\emph{liquid}, that fills the segment $[a(t),b(t)]$ and is surrounded
by another compressible fluid, say the \emph{gas}, filling the rest of
the real line (see Figure~\ref{fig:duefluidi}). We assume that the gas
obeys a fixed pressure law $\overline P_{g}\left(\rho\right)$, while
for the liquid we assume a one parameter family of pressure laws
$\overline P_{\kappa}\left(\rho\right)$ such that $\overline
P_{\kappa}^{\prime}\left(\rho\right)\to +\infty$ as $\kappa\to 0$. The
total mass of the liquid is fixed:
$\int_{a(t)}^{b(t)}\rho\left(t,x\right)\; \d{x} = m$.%
\definecolor{wwzzff}{rgb}{0.4,0.6,1.0}%
\definecolor{zzttqq}{rgb}{0.6,0.2,0.0} \definecolor{cqcqcq}{rgb}
{0.7529411764705882,0.7529411764705882,0.7529411764705882}%
\begin{figure}[!ht]
  \centering%
  \begin{tikzpicture}[line cap=round,line join=round,>=triangle
    45,x=0.45cm,y=0.45cm]%
    \clip(-4.300000000000001,-4.040000000000001) rectangle
    (20.0000000000003,0.5); \fill[color=zzttqq,fill=zzttqq,fill
    opacity=0.8] (3.0,-1.0) -- (3.0,-2.0) -- (9.0,-2.0) -- (9.0,-1.0)
    -- cycle; \fill[color=wwzzff,fill=wwzzff,fill opacity=0.75]
    (-3.0,-1.0) -- (-3.0,-2.0) -- (3.0,-2.0) -- (3.0,-1.0) -- cycle;
    \fill[color=wwzzff,fill=wwzzff,fill opacity=0.75] (9.0,-1.0) --
    (9.0,-2.0) -- (19.0,-2.0) -- (19.0,-1.0) -- cycle; \draw [line
    width=1.5000000000000003pt] (-3.0,-1.0)-- (19.0,-1.0); \draw [line
    width=1.5000000000000003pt] (19.0,-2.0)-- (-3.0,-2.0); \draw [line
    width=0.5000000000000003pt,<-] (20.0,-1.5)-- (-4.0,-1.5);
    \draw[color=black] (-0.74,-2.900000000000002) node
    {$\overline{P}_{g}(\rho)$}; \draw[color=black]
    (15.74,-2.900000000000002) node {$\overline{P}_{g}(\rho)$};
    \draw[color=black] (5.74,-2.900000000000002) node
    {$\overline{P}_{\kappa}(\rho)$}; \draw[color=black]
    (2.74,-0.3000000000002) node {$a(t)$}; \draw[color=black]
    (9.0,-0.3000000000002) node {$b(t)$}; \draw[color=black]
    (19.5,-2.2000000000002) node {$x$}; \draw[color=black] (14,0) node
    {``gas''}; \draw[color=black] (6,0) node {``liquid''};
    \draw[color=black] (0,0) node {``gas''};
  \end{tikzpicture}%
  \caption{The two immiscible fluids: the liquid is in the middle,
    while the gas fills the two sides.}%
  \label{fig:duefluidi}
\end{figure}
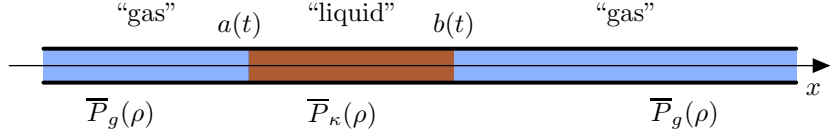%
Since the two fluids are immiscible, the introduction of the
Lagran\-gian coordinate $z$ and of the specific volume $\tau$ is a
natural choice~\cite{Wagner}:
\begin{equation}
  \label{eq:lagcoord}
  z\left(t,x\right)
  =
  \int_{a(t)}^{x}\rho\left(t,\xi\right)\; \d\xi,
  \qquad
  \tau = \frac{1}{\rho},
  \qquad
  P_{g}\left(\tau\right)=\overline{P}_{g}\left(\frac{1}{\tau}\right),
  \qquad P_{\kappa}\left(\tau\right)=\overline{P}_{\kappa}\left(\frac{1}{\tau}\right) .
\end{equation}
In these coordinates, the liquid and gas phases become the fixed sets
(see Figure~\ref{fig:liquidogas2})
\begin{displaymath}
  \mathcal{L} = \left]0,m \right[
  \qquad \mbox{ and } \qquad
  \mathcal{G} = \reali \setminus \left]0,m \right[ \,.
\end{displaymath}
\begin{figure}[!ht]
  \centering \definecolor{wwzzff}{rgb}{0.4,0.6,1.0}
  \definecolor{zzttqq}{rgb}{0.6,0.2,0.0}
  \definecolor{cqcqcq}{rgb}{0.7529411764705882,0.7529411764705882,0.7529411764705882}
  \begin{tikzpicture}[line cap=round,line join=round,>=triangle
    45,x=0.45cm,y=0.45cm]
    \clip(-4.00000000000001,-4.00000000000001) rectangle
    (20.0000000000003,-0.0); \fill[color=zzttqq,fill=zzttqq,fill
    opacity=0.8] (3.0,-1.0) -- (3.0,-2.0) -- (9.0,-2.0) -- (9.0,-1.0)
    -- cycle; \fill[color=wwzzff,fill=wwzzff,fill opacity=0.75]
    (-3.0,-1.0) -- (-3.0,-2.0) -- (3.0,-2.0) -- (3.0,-1.0) -- cycle;
    \fill[color=wwzzff,fill=wwzzff,fill opacity=0.75] (9.0,-1.0) --
    (9.0,-2.0) -- (19.0,-2.0) -- (19.0,-1.0) -- cycle; \draw [line
    width=1.5000000000000003pt] (-3.0,-1.0)-- (19.0,-1.0); \draw [line
    width=1.5000000000000003pt] (19.0,-2.0)-- (-3.0,-2.0);
    \draw[color=black] (-0.74,-2.900000000000002) node
    {$P_{g}(\tau)$}; \draw[color=black] (15.74,-2.900000000000002)
    node {$P_{g}(\tau)$}; \draw[color=black] (5.74,-2.900000000000002)
    node {$P_{\kappa}(\tau)$}; \draw[color=black]
    (3.0,-0.3000000000002) node {$0$}; \draw[color=black]
    (9.0,-0.3000000000002) node {$m$}; \draw [line
    width=0.5000000000000003pt,<-] (20.0,-1.5)-- (-4.0,-1.5);
    \draw[color=black] (19.5,-2.2000000000002) node {$z$}; \draw [line
    width=0.5000000000000003pt] (3.00,-2.5)-- (3.00,-0.7); \draw [line
    width=0.5000000000000003pt] (9.00,-2.5)-- (9.00,-0.7);
  \end{tikzpicture}
  \caption{In Lagrangian coordinates, the boundaries separating the
    two fluids are fixed.}
  \label{fig:liquidogas2}
\end{figure}
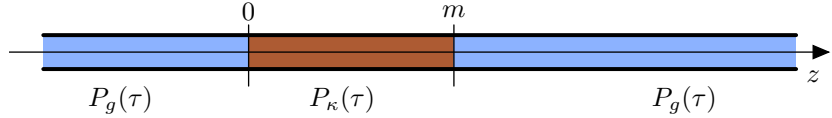

On $P_{g}\left(\tau\right)$ and $P_{\kappa}\left(\tau\right)$, we
require the usual hypotheses and the incompressible limit assumption:
\begin{equation}
  \label{eq:Pproperties}
  P_{g}\left(\tau\right),\, P_{\kappa}\left(\tau\right)>0;\;
  P'_{g}\left(\tau\right),\, P_{\kappa}^{\prime}\left(\tau\right)<0;\;
  P''_{g}\left(\tau\right),\, P_{\kappa}^{\prime\prime}\left(\tau\right)>0;\;
  P_{\kappa}^{\prime}\left(\tau\right)\xrightarrow{\kappa\to 0}-\infty \,.
\end{equation}
In the isentropic approximation, the dynamics of the two fluids is
described by the $p$-system~\cite{DafermosBook}
\begin{equation}
  \label{eq:baseeqint}
  \begin{cases}
    \partial_{t}\tau-\partial_{z}v=0
    \\
    \partial_{t}v+\partial_{z}P_{\kappa}\left(z,\tau\right)=0,
  \end{cases}
  \quad \mbox{ where } \quad
  P_{\kappa}\left(z,\tau\right)=
  \begin{cases}
    P_{\kappa}\left(\tau\right)&\mbox{ for }z\in
    \mathcal{L}\\
    P_{g}\left(\tau\right)&\mbox{ for }z\in \mathcal{G},
  \end{cases}
\end{equation}
$v(t,z)$ being the speed of the fluids at time $t$ and at the
Lagrangian coordinate $z$.

The Rankine--Hugoniot conditions for~\eqref{eq:baseeqint}, applied at
$z=0$ and $z=m$, imply the following interface conditions
(conservation of mass and momentum) for a.e. $t\geq 0$:
\begin{displaymath}
  \begin{cases}
    v\left(t,0-\right)=v\left(t,0+\right)
    \\
    P_{g}\left(\tau\left(t,0-\right)\right) =
    P_{\kappa}\left(\tau\left(t,0+\right)\right),
  \end{cases}
  \quad
  \begin{cases}
    v\left(t,m-\right)=v\left(t,m+\right)
    \\
    P_{\kappa}\left(\tau\left(t,m-\right)\right) =
    P_{g}\left(\tau\left(t,m+\right)\right).
  \end{cases}
\end{displaymath}
In other words, the pressure and the velocity have to be continuous
across the interfaces. Hence, the pressure is a natural choice as
unknown, rather than the specific volume. Therefore, we introduce the
inverse functions of the pressure laws
\begin{equation}
  \label{eq:Tproperties}
  \mathcal{T}_{g}(p)=P^{-1}_{g}\left(p\right),\quad
  \mathcal{T}_{\kappa}(p)=P^{-1}_{\kappa}\left(p\right),\quad
  \mathcal{T}_{\kappa}'\left(p\right)\xrightarrow{\kappa\to 0} 0 \,,
\end{equation}
the last limit being a consequence of~\eqref{eq:Pproperties}. Rewrite
system~\eqref{eq:baseeqint} with $(p,v)$ as unknowns
\begin{equation}
  \label{eq:baseeqinp}
  \begin{cases}
    \partial_{t}\mathcal{T}_{\kappa}\left(z,p\right)-\partial_{z}v=0
    \\
    \partial_{t}v+\partial_{z}p=0 \,,
  \end{cases}
  \quad \mbox{ where } \quad
  \mathcal{T}_{\kappa}\left(z,p\right)=
  \begin{cases}
    \mathcal{T}_{\kappa}\left(p\right)&\mbox{ for }z\in \mathcal{L}
    \\
    \mathcal{T}_{g}\left(p\right)&\mbox{ for }z\in\mathcal{G} \,.
  \end{cases}
\end{equation}
The conditions at the interfaces become continuity requirements on the
unknown functions:
\begin{equation}
  \label{eq:intCond}
  \begin{cases}
    v\left(t,0-\right)=v\left(t,0+\right)
    \\
    p\left(t,0-\right)=p\left(t,0+\right)
  \end{cases}
  \qquad
  \begin{cases}
    v\left(t,m-\right)=v\left(t,m+\right)
    \\
    p\left(t,m-\right)=p\left(t,m+\right)
  \end{cases}\quad \mbox{for a.e. }t\geq 0 \,.
\end{equation}
The choice of these unknowns significantly simplifies the study of the
Riemann problem at the interfaces.

\bigskip

Particular care is necessary to select the one parameter family of
pressure laws, the main constraint being the validity
of~\eqref{eq:pfamprop} for all $\kappa$. Indeed, \eqref{eq:pfamprop}
ensures that we recover the same equations obtained through scaling
and studied in~\cite{KlainermanMajda1981, KlainermanMajda1982}. The
family $\overline P_{\kappa}\left(\rho\right) =\frac{1}{\kappa^{2}}
\overline P\left(\rho\right)$ chosen in~\cite{KlainermanMajda1981}
diverges to $+\infty$ as $\kappa\to 0$. This is not a problem when
studying only one fluid as in~\cite{KlainermanMajda1981} because the
pressure enters the equations only through its gradient. In our case,
the value of the pressure is very relevant, since it enters the
interface conditions~\eqref{eq:intCond}. Therefore, we cannot allow
the pressure to grow nonphysically to $+\infty$. We fix the density
$\bar \rho$ of the incompressible fluid in the limit and impose that
the pressure at that particular density $\bar \rho$ is a constant,
independent of $\kappa$:
\begin{equation}
  \label{eq:pressurecond2}
  \overline P_{\kappa}(\bar \rho)
  =
  \bar p,\quad \mbox{ for all }\kappa\in \left]0,1\right[ \,.
\end{equation}
Now, choose a fixed pressure law $ \overline P = \overline P(\rho)$
(for instance, an admissible choice is the usual $\gamma$--law
$\overline P\left(\rho\right) = k \, \rho^{\gamma}$ with $\gamma\geq
1$) and apply conditions~\eqref{eq:pfamprop}
and~\eqref{eq:pressurecond2} to get the following expression for
$\overline P_{\kappa}\left(\rho\right)$, with $\overline{P}\left(\bar
  \rho\right)=\bar p$:
\begin{equation}
  \label{eq:2}
  \overline
  P_{\kappa}\left(\rho\right)
  =
  \bar p
  +
  \frac{1}{\kappa^{2}}\left[
    \overline{P}\left(\rho\right)-\bar p\right] \,,
\end{equation}
which, with the substitution $\rho=\frac{1}{\tau}$, becomes:

\begin{equation}
  \label{eq:2bis}
  P_{\kappa}\left(\tau\right) = \bar p + \frac{1}{\kappa^{2}}\left[
    P\left(\tau\right)-\bar p\right].
\end{equation}

Finally, in term of the inverse functions
$\mathcal{T}_{\kappa}=P_{\kappa}^{-1}$ and $\mathcal{T}=P^{-1}$, we
have
\begin{equation}
  \label{eq:tkfamily}
  \mathcal{T}_{\kappa}\left(p\right) =
  \mathcal{T}\left(\bar p + \kappa^{2}
    \left(p-\bar p\right)\right)\,,\quad
  \lim_{\kappa\to
    0}\mathcal{T}_{\kappa}\left(p\right)=\mathcal{T}\left(\bar
    p\right)=\frac{1}{\bar \rho}\dot = \bar \tau\,.
\end{equation}
In~\cite{ColomboGuerraSchleper, CGSInc}, \eqref{eq:tkfamily} is
approximated linearly:
\begin{equation}
  \label{eq:linear}
  \mathcal{T}\left(\bar p + \kappa^{2} \left(p-\bar p\right)\right)
  \approx
  \mathcal{T}\left(\bar p\right)
  + \kappa^{2}\mathcal{T}'\left(\bar p\right)
  \left(p-\bar p\right)
  =
  \bar \tau
  + \kappa^{2}\mathcal{T}'\left(\bar p\right)
  \left(p-\bar p\right)\,,
\end{equation}
so that the liquid phase turns out to be governed by a linear
system. This approximation makes all the estimates simpler.  Here we
study the Cauchy problem in the fully non linear case
\begin{equation}
  \label{eq:finalequation}
  \begin{cases}
    \partial_{t}\mathcal{T}_{\kappa}\left(z,p\right)-\partial_{z}v=0
    \\
    \partial_{t}v+\partial_{z}p=0,
  \end{cases}
  \quad \mbox{ where } \quad
  \mathcal{T}_{\kappa}\left(z,p\right)=
  \begin{cases}
    \mathcal{T}\left(\bar p + \kappa^{2} \left(p-\bar p\right)\right)
    &\mbox{ for }z\in
    \mathcal{L}\\
    \mathcal{T}_{g}\left(p\right)&\mbox{ for }z\in \mathcal{G}.
  \end{cases}
\end{equation}
Colombo and Schleper in~\cite[Theorem~2.5]{ColomboSchleper} proved
that for any fixed small $\kappa>0$, there exists a Lipschitz
semigroup of solutions to~\eqref{eq:finalequation}, but their
estimates are not uniform with respect to $\kappa$. Therefore, as
$\kappa\to 0$ the Lipschitz constant of the semigroup could blow up
and its domain could shrink, becoming trivial. Here, we provide a full
set of new estimates either uniform in $\kappa$, or with the
dependence on $\kappa$ made explicit. To this aim, we substantially
improve the wave front tracking construction
in~\cite{BressanLectureNotes, ColomboSchleper}, devising and
exploiting a different parametrization of the Lax curves.

The main result of this paper states the rigorous convergence at the
incompressible limit in the liquid phase of the solutions
to~\eqref{eq:finalequation} to solutions to
\begin{equation}
  \label{eq:IncompL}
  \left\{
    \begin{array}{lr@{\,}c@{\,}l@{\qquad\qquad}l@{}}
      \left\{
        \begin{array}{@{}l}
          \partial_t \mathcal{T}_{g}(p) - \partial_z v = 0
          \\
          \partial_t v
          +
          \partial_z  p
          =
          0
        \end{array}
      \right.
      &
      z & \in &  \mathcal{G}
      &
      \mbox{gas;}
      \\[20pt]
      \dot v
      =
      \frac{p(t, 0-) - p(t, m+)}{m}
      & & & &
      \mbox{liquid;}
      \\[10pt]
      \left\{
        \begin{array}{@{}rcl@{}}
          v \left(t, 0-\right)
          & = &
          v (t)
          \\
          v \left(t, m+\right)
          & = &
          v (t)
        \end{array}
      \right.
      &
      & &
      &
      \begin{array}{@{}l}
        \mbox{immiscibility and mass }
        \\
        \mbox{conservation.}
      \end{array}
    \end{array}
  \right.
\end{equation}
Note that the liquid speed $v (t)$ is independent of the Lagrangian
variable $z$. In this very singular limit, the sound speed in the
liquid phase tends to $+\infty$; the density converges to a fixed
reference value $\bar\rho$; the graph of the pressure law
$P_{\kappa}\left(\tau\right)$ becomes vertical and the eigenvectors of
the Jacobian of the flow tend to coalesce. Moreover, the pressure in
the liquid wildly oscillates but, remarkably, we are able to prove the
weak$^\star$ convergence of the pressure to the linear interpolation
of the traces of the pressure at the sides of the liquid region, as is
to be expected based on physical considerations. A linear example,
where all the components of this singular limit can be explicitly
computed, can be found in~\cite{ColomboGuerraInc}.

Recall that problem~\eqref{eq:finalequation},
respectively~\eqref{eq:IncompL}, is well posed in $\L1$ globally in
time, see~\cite[Theorem~2.5]{ColomboSchleper},
respectively~\cite[Theorem~3.6]{BorscheColomboGaravello3}.

\section{Main Result}
\label{sec:Main}

Throughout, we require that the pressure law $P_g$ in the gas phase
and the one parameter family of pressure laws $P_\kappa$ in the liquid
phase, as defined in~\eqref{eq:2bis}, all satisfy the condition
\begin{description}
\item[(P):] $P \in \C3 (\left]0, +\infty\right[; \left]0,
    +\infty\right[)$, $P' < 0$ and $P'' > 0$.
\end{description}
\noindent The standard choice $p (\tau) = k / \tau^\gamma$ satisfies
this condition for all $k>0$ and $\gamma \geq 1$.

As a starting point, we provide the rigorous definition of solutions
to~\eqref{eq:baseeqinp}, with reference to~\cite[Chapter~4,
Definition~4.3 and Admissibility Condition~2]{BressanLectureNotes}.

\begin{definition}
  \label{def:solution}
  Fix $T > 0$ and $\kappa > 0$. By \emph{weak solution}
  to~\eqref{eq:baseeqinp} we mean a map
  \begin{displaymath}
    (p,v)
    \in
    \C0 \left([0,T]; (\Lloc1 \cap \BV) (\reali; \reali^+
      \times \reali)\right)
  \end{displaymath}
  such that~\eqref{eq:baseeqinp} holds in distributional sense. The
  weak solution $u$ is a \emph{weak entropy solution}
  to~\eqref{eq:baseeqinp} if both its restrictions to $\mathcal{L}$
  and to $\mathcal{G}$ are weak entropy solutions in the sense
  of~\cite[Definition~4.3]{BressanLectureNotes}.
\end{definition}

\noindent Introduce the mathematical entropy flow $q = p v$
of~\eqref{eq:baseeqinp}, the equalities $(p,v) (t, 0-) = (p,v) (t,
0+)$ and $(p,v) (t, m-) = (p,v) (t, m+)$ (consequences of the
Rankine--Hugoniot conditions) imply that the entropy flow is
continuous and hence that the entropy is conserved across both
interfaces.

In the case of~\eqref{eq:IncompL}, we
recall~\cite[Definition~2.5]{BorscheColomboGaravello2}.

\begin{definition}
  \label{def:sol2}
  Fix $T>0$. By a \emph{solution} to~\eqref{eq:IncompL} we mean a pair
  of maps
  \begin{eqnarray*}
    (p^{*},v^{*})
    & \in &
    \C0\left(
      [0,T];
      (\Lloc1 \cap \BV) (\mathcal{G}; \reali^+ \times \reali)
    \right)
    \\
    v_{l}
    & \in &
    \W{1}{\infty} ([0,T]; \reali)
  \end{eqnarray*}
  such that:
  \begin{enumerate}
  \item $(p^{*},v^{*})$ is a weak entropy solution to $\left\{
      \begin{array}{l}
        \partial_t \mathcal{T}_{g}(p) - \partial_z v = 0
        \\
        \partial_t v + \partial_z p = 0
      \end{array}
    \right.$ in $[0,T] \times \mathcal{G}$;
  \item for a.e.~$t \in [0,T]$, $\dot v_{l} (t) = \frac{1}{m}
    \left(p^{*}(t, 0-) - p^{*} (t, m+)\right)$;
  \item for a.e.~$t \in [0,T]$, $v^{*} (t,0-) = v^{*} (t,m+) = v_{l}
    (t)$.
  \end{enumerate}
\end{definition}

\noindent The existence of solutions to~\eqref{eq:finalequation}
follows from the next theorem, that also provides the basic estimates
for the subsequent compressible to incompressible limit. In this
context, a natural requirement is the smallness of the total variation
of the initial datum. Aiming at the incompressible limit, it is
natural to introduce the weighted total variation
\begin{equation}
  \label{eq:wtv}
  {\wtv}_{\kappa}\left(p,v\right)
  =
  \tv\left(p,\reali\right)
  +
  \tv\left(v,\mathcal{G}\right)
  +
  \frac{1}{\kappa} \; \tv\left(v,\mathcal{L}\right)
\end{equation}
whose boundedness requires that the initial total variation of the
particles speed in the liquid vanishes with $\kappa$.

\begin{theorem}
  \label{thm:kappa}
  Fix the total mass of the liquid $m > 0$ and a pressure $p_{o} >
  0$. Let $P, P_g$ satisfy~\textbf{(P)}, define $\mathcal{T}_{g}$ as
  in~\eqref{eq:Tproperties} and $\mathcal{T}_{\kappa}$ as
  in~\eqref{eq:tkfamily}. Then, there exist positive
  $\delta,\;\Delta,\;L,\;\kappa_{*}$ with $\kappa_{*}<1$ such that for
  any $\kappa \in \left]0, \kappa_*\right[$, for any initial datum
  $(\tilde p, \tilde v) \in \Lloc1 (\reali; \reali^+ \times \reali)$,
  under the assumptions
  \begin{equation}
    \label{eq:delta1}
    {\wtv}_\kappa (\tilde p, \tilde v)
    \leq
    \delta
    \,, \quad
    \norma{\tilde p - p_{o}}_{\L\infty (\reali;\reali)}
    \leq
    \delta
    \,, \quad
    \tilde v (0+) = \tilde v (0)
    \quad \mbox{ and } \quad
    \tilde v (m-) = \tilde v (m) \,,
  \end{equation}
  problem~\eqref{eq:finalequation} with initial datum $(\tilde
  p,\tilde v)$ admits a weak entropy solution $(p^{\kappa},
  v^{\kappa})$ in the sense of Definition~\ref{def:solution} defined
  for all $t \in \reali^+$. Moreover, since the specific volume is
  $\tau^{\kappa}(t,z) = \mathcal{T}_{\kappa}\left(z,p^{\kappa}(t,z)\right)$,\\
  for any $t,t_{1},t_{2}\geq 0$
  \begin{equation}
    \label{eq:THMtimeestimates}
    \begin{array}{@{}l@{\,}c@{\,}rl@{\,}c@{\,}r@{}}
      {\wtv}_{\kappa} \left(u^{\kappa}(t,\cdot)\right)
      & \leq &
      \Delta,
      &&
      \\[10pt]
      \tv\left(p^{\kappa}(t,\cdot), \mathcal{L}\right)
      & \leq &
      \Delta,
      &
      \int_{\mathcal{L}} \left|p^{\kappa}
        (t_{2},z)-p^{\kappa}(t_{1},z)\right| \d{z}
      & \leq &
      \frac{1}{\kappa}L\left|t_{2}-t_{1}\right|,
      \\[10pt]
      \tv\left(v^{\kappa}(t,\cdot),\mathcal{L}\right)
      & \leq &
      \kappa\Delta,
      &\int_{\mathcal{L}}\left|v^{\kappa}
        (t_{2},z)-v^{\kappa}(t_{1},z)\right| \d{z}
      & \leq &   L\left|t_{2}-t_{1}\right|,
      \\[10pt]
      \tv\left(\tau^{\kappa}(t,\cdot), \mathcal{L}\right)
      & \leq &
      \kappa^{2}\Delta,
      &\int_{\mathcal{L}}\left|\tau^{\kappa}
        (t_{2},z)-\tau^{\kappa}(t_{1},z)\right| \d{z}
      & \leq &
      \kappa L\left|t_{2}-t_{1}\right|,
      \\[10pt]
      \tv \left(p^{\kappa}(t,\cdot), \mathcal{G}\right)
      & \leq &
      \Delta,
      &\int_{\mathcal{G}}
      \left|p^{\kappa}(t_{2},z)-p^{\kappa}(t_{1},z)\right| \d{z}
      & \leq &
      L\left|t_{2}-t_{1}\right|,
      \\[10pt]
      \tv \left(v^{\kappa}(t,\cdot),\mathcal{G}\right)
      & \leq &
      \Delta,
      &
      \int_{\mathcal{G}}\left|v^{\kappa}
        (t_{2},z)-v^{\kappa}(t_{1},z)\right| \d{z}
      & \leq &
      L\left|t_{2}-t_{1}\right|,
      \\[10pt]
      \tv\left(\tau^{\kappa}(t,\cdot),\mathcal{G}\right)
      & \leq &
      \Delta,
      &
      \int_{\mathcal{G}}\left|\tau^{\kappa}
        (t_{2},z)-\tau^{\kappa}(t_{1},z)\right| \d{z}
      & \leq &
      L\left|t_{2}-t_{1}\right|;
      \\[10pt]
    \end{array}
  \end{equation}
  for any $z,z_{1},z_{2}\in \mathcal{L}$
  \begin{equation}
    \label{eq:THMspaceestimates}
    \begin{array}{@{}l@{\,}c@{\,}rl@{\,}c@{\,}r@{}}
      \tv\left(p^{\kappa}(\cdot,z),\reali^{+}\right)
      & \leq &
      \frac{\Delta}{\kappa},
      &\int_{\reali^+}\left|p^{\kappa}
        (t,z_{2})-p^{\kappa}(t,z_{1})\right|\d{t}
      & \leq &
      L\left|z_{2}-z_{1}\right|,
      \\[10pt]
      \tv\left(v^{\kappa}(\cdot,z),\reali^{+}\right)
      & \leq &
      \Delta,
      &\int_{\reali^+}\left|v^{\kappa}
        (t,z_{2})-v^{\kappa}(t,z_{1})\right|\d{t}
      & \leq &
      \kappa L\left|z_{2}-z_{1}\right|,
      \\[10pt]
      \tv\left(\tau^{\kappa}(\cdot,z),\reali^{+}\right)
      & \leq &
      \kappa\Delta,
      &\int_{\reali^+}\left|\tau^{\kappa}
        (t,z_{2})-\tau^{\kappa}(t,z_{1})\right|\d{t}
      & \leq &
      \kappa^{2}L\left|z_{2}-z_{1}\right|;
      \\[10pt]
    \end{array}
  \end{equation}
  for any $z,z_{1},z_{2}\in \mathcal{G}$
  \begin{equation}
    \label{eq:THMspaceestimates2}
    \begin{array}{@{}l@{\,}c@{\,}rl@{\,}c@{\,}r@{}}
      \tv\left(p^{\kappa}(\cdot,z),\reali^{+}\right)
      & \leq &
      \Delta,
      &\int_{\reali^+}\left|p^{\kappa}
        (t,z_{2})-p^{\kappa}(t,z_{1})\right|\d{t}
      & \leq &
      L\left|z_{2}-z_{1}\right|,
      \\[10pt]
      \tv\left(v^{\kappa}(\cdot,z),\reali^{+}\right)
      & \leq &
      \Delta,
      &\int_{\reali^+}\left|v^{\kappa}
        (t,z_{2})-v^{\kappa}(t,z_{1})\right|\d{t}
      & \leq &
      L\left|z_{2}-z_{1}\right|,
      \\[10pt]
      \tv\left(\tau^{\kappa}(\cdot,z),\reali^{+}\right)
      & \leq &
      \Delta,
      &\int_{\reali^+}\left|\tau^{\kappa}
        (t,z_{2})-\tau^{\kappa}(t,z_{1})\right|\d{t}
      & \leq &
      L\left|z_{2}-z_{1}\right|;
    \end{array}
  \end{equation}
  for any $z,z_{1},z_{2}\in \reali$
  \begin{equation}
    \label{eq:spaceestimatesAllLine}
    \begin{array}{@{}l@{\,}c@{\,}rl@{\,}c@{\,}r@{}}
      \tv\left(p^{\kappa}(\cdot,z),\reali^{+}\right)
      & \leq &
      \frac{\Delta}{\kappa},
      &\int_{\reali^+}\left|p^{\kappa}
        (t,z_{2})-p^{\kappa}(t,z_{1})\right|\d{t}
      & \leq &
      \frac{L}{\kappa}\left|z_{2}-z_{1}\right|,
      \\[10pt]
      \tv\left(v^{\kappa}(\cdot,z),\reali^{+}\right)
      & \leq &
      \Delta,
      &\int_{\reali^+}\left|v^{\kappa}
        (t,z_{2})-v^{\kappa}(t,z_{1})\right|\d{t}
      & \leq &
      L\left|z_{2}-z_{1}\right|.
    \end{array}
  \end{equation}
\end{theorem}

The above existence result can be completed with uniqueness and
Lipschitz continuous dependence of the solutions on the data
exploiting the results in~\cite[Theorem~2.5]{ColomboSchleper}. Note
however that the estimates provided therein, differently from the ones
presented here, are not uniform in $\kappa$.

We now pass to the key limit $\kappa \to 0$.

\begin{theorem}
  \label{thm:limit}
  Fix the total mass of the liquid $m>0$ and a pressure $p_{o}
  >0$. Let $P, P_g$ satisfy~\textbf{(P)}, define $\mathcal{T}_{g}$ as
  in~\eqref{eq:Tproperties} and $\mathcal{T}_{\kappa}$ as
  in~\eqref{eq:tkfamily}. Let $\delta$, $\Delta$, $L$ and $\kappa_{*}$
  be as in Theorem~\ref{thm:kappa}. For any $v_o \in \reali$ and
  $(\tilde p, \tilde v) \in \Lloc1 (\reali; \reali^+ \times \reali)$
  satisfying
  \begin{equation}
    \label{eq:questa}
    \norma{\tilde p - p_o}_{\L\infty (\reali; \reali)} < \delta \,,
    \qquad
    \tv\left(\tilde p\right)+\tv\left(\tilde v\right)\leq \delta
    \quad \mbox{ and } \quad
    \tilde v(z) = v_{o} \quad\forall z\in [0,m] \,,
  \end{equation}
  the Cauchy problem for~\eqref{eq:finalequation} with initial datum
  $(\tilde p, \tilde v)$ admits for any $\kappa \in \left]0,
    \kappa^*\right[$ a weak entropy solution
  $\left(p^{\kappa},v^{\kappa}\right)$
  satisfying~\eqref{eq:THMtimeestimates} --
  \eqref{eq:THMspaceestimates} -- \eqref{eq:THMspaceestimates2} --
  \eqref{eq:spaceestimatesAllLine}.

  Moreover, there exist functions
  \begin{displaymath}
    \begin{array}{rcl@{\qquad\qquad}rcl}
      p^{*} & \in &
      \C0\left(\reali^+; (\Lloc1 \cap \BV) (\mathcal{G}; \reali^+)\right),
      &
      p_l & \in &
      \L\infty (\reali^+ \times \mathcal{L}; \reali^+),
      \\
      v^{*} & \in &
      \C0\left(\reali^+; (\Lloc1 \cap \BV) (\reali; \reali^+)\right),
      &
      v_l & \in &
      \W{1}{\infty} (\reali^+; \reali),
    \end{array}
  \end{displaymath}
  such that $(p^*, v^*_{|\mathcal{G}})$ and $v_l$
  solve~\eqref{eq:IncompL} with initial datum
  \begin{eqnarray*}
    (p^*, v^*) (0,z) & = & (\tilde p, \tilde v) (z)
    \quad \mbox{ a.e. } z \in \mathcal{G}
    \\
    v_l (0) & = & v_o
  \end{eqnarray*}
  in the sense of Definition~\ref{def:sol2}. Up to subsequences, as
  $\kappa\to 0$,
  \begin{equation}
    \label{eq:vconvergence}
    \begin{array}{@{}l@{}}
      \begin{array}{@{}rclrll}
        v^{\kappa}\left(t,\cdot\right)
        & \to
        & v^{*}\left(t,\cdot\right)
        & \mbox{ in }
        & \Lloc1\left(\reali; \reali\right),
        & t\geq 0
        \\[3pt]
        v^{\kappa}\left(\cdot,z\right)
        & \to
        & v^{*}\left(\cdot,z\right)
        & \mbox{ in }
        & \Lloc1(\reali^{+}; \reali),
        & z\in\reali
      \end{array}
      \\[12pt]
      \begin{array}{@{}lcllcll@{}}
        \tv \left(v^{*}(t,\cdot),\reali\right)
        & \leq &
        \Delta,
        &
        \int_{\reali}\left|v^{*}
          (t_{2},z)-v^{*}(t_{1},z)\right| \d{z}
        & \leq &
        L\left|t_{2}-t_{1}\right|,
        & t,t_{1},t_{2}\geq 0,
        \\[6pt]
        \tv\left(v^{*}(\cdot,z),\reali^{+}\right)
        & \leq &
        \Delta,
        &
        \int_{\reali^+}\left|v^{*}
          (t,z_{2})-v^{*}(t,z_{1})\right|\d{t}
        & \leq &
        L\left|z_{2}-z_{1}\right|,
        & z,z_{1},z_{2}\in\reali,
      \end{array}
      \\[12pt]
      v^{*}\left(t,z\right) = v_{l}(t),
      \mbox{ a.e. }(t,z)\in \reali^+ \times \mathcal{L}.
    \end{array}
  \end{equation}
  \begin{equation}
    \label{eq:pconvergence}
    \begin{array}{@{}l@{}}
      \begin{array}{@{}rclrll}
        p^{\kappa}\left(t,\cdot\right)
        & \to
        & p^{*}\left(t,\cdot\right)
        & \mbox{ in }
        & \Lloc1(\mathcal{G}; \reali),
        & t\geq 0
        \\[3pt]
        p^{\kappa}\left(\cdot,z\right)
        & \to
        & p^{*}\left(\cdot,z\right)
        & \mbox{ in }
        & \Lloc1(\reali^{+}; \reali),
        & z\in\mathcal{G}
      \end{array}
      \\[12pt]
      \begin{array}{@{}lcllcll@{}}
        \tv \left(p^{*}(t,\cdot),\mathcal{G}\right)
        & \leq &
        \Delta,
        &
        \int_{\mathcal{G}}\left|p^{*}
          (t_{2},z)-p^{*}(t_{1},z)\right| \d{z}
        & \leq &
        L\left|t_{2}-t_{1}\right|,
        & t,t_{1},t_{2}\geq 0,
        \\[6pt]
        \tv\left(p^{*}(\cdot,z),\reali^{+}\right)
        & \leq &
        \Delta,
        &
        \int_{\reali^+}\left|p^{*}
          (t,z_{2})-p^{*}(t,z_{1})\right|\d{t}
        & \leq &
        L\left|z_{2}-z_{1}\right|,
        & z,z_{1},z_{2}\in\mathcal{G},
      \end{array}
      \\[12pt]
      p^{\kappa}(\cdot,\cdot) \quad \wsto \quad  p_{l}(\cdot,\cdot),\quad
      \mbox{ in } \quad \L\infty(\mathcal{L}\times\reali^{+}; \reali^+),
      \\[3pt]
      p_{l}(t,z) \quad = \quad
      \left(1-\frac{z}{m}\right) p^{*}(t,0-) + \frac{z}{m}p^{*}(t,m+),
      \mbox{ a.e. } (t,z) \in \reali^+ \times \mathcal{L} \,,
      \\[3pt]
      \tau^{\kappa}\left(\cdot,\cdot\right) \quad \to \quad \bar \tau,
      \mbox{ uniformly  in }
      \mathcal{L}\times \reali^{+}.
    \end{array}
  \end{equation}
  where the specific volume is $\tau^{\kappa}(t,z) =
  \mathcal{T}_{\kappa} \left(z,p^{\kappa}(t,z)\right)$.
\end{theorem}

\par From the Eulerian coordinates' point of view, the locations of
the boundaries of the liquid phase can be recovered through a time
integration. Let $x = a_{o}$ be the initial location of the left
interface that we keep fixed with respect to $\kappa$.  Since in
Theorem~\ref{thm:limit} the initial pressure is chosen independently
of $\kappa$, the initial specific volume in the liquid is given by
$\tilde \tau^{\kappa}(z) = \mathcal{T}\left(\bar p +
  \kappa^{2}\left(\tilde p(z) - \bar p\right)\right)$, which may
depend on $\kappa$. The total mass $m$ of the liquid is fixed. Hence,
the initial location of the right interface in general depends on
$\kappa$, say $x=b_{o}^{\kappa}$. Since $\tilde \tau^{\kappa}(z)\to
\bar \tau$ as $\kappa\to 0$, we have $b_{o}^\kappa\to
b_{o}=a_{o}+m\bar \tau$. Note however that in the particular case of
constant initial pressure $\tilde p(z)=\bar p$ in the liquid, also
$b_o^\kappa$ turns out to be independent of $\kappa$.

Let $a^{\kappa}(t)$ and $b^{\kappa}(t)$ be the locations of the
interfaces (in Eulerian coordinates) at time $t$ for positive
$\kappa$, while $a(t)$ and $b(t)$ be the corresponding limits as
$\kappa \to 0$. Then, we have:
\begin{equation}
  \label{eq:eulerinterfaces}
  \begin{array}{rcl@{\qquad\qquad}rcl}
    a^{\kappa}(t) & = &
    \displaystyle
    a_{o} + \int_{0}^{t}v^{\kappa}\left(\xi,0\right)\; \d\xi
    &
    a(t) & = &
    \displaystyle
    a_{o} + \int_{0}^{t}v_{l}\left(\xi\right)\;d\xi
    \\[5pt]
    b^{\kappa}(t) & = &
    \displaystyle
    b_{o}^\kappa + \int_{0}^{t}v^{\kappa}\left(\xi,m\right) \; \d\xi
    &
    b(t) & = &
    \displaystyle
    b_{o} + \int_{0}^{t}v_{l}\left(\xi\right)\; \d\xi\,.
  \end{array}
\end{equation}
Using Theorem~\ref{thm:limit} we can see that the boundaries of the
two phases are Lipschitz continuous functions of $t$. Moreover, as
$\kappa \to 0$, $a^{\kappa}\to a$ and $b^{\kappa}\to b$ uniformly on
bounded time intervals. An explicit expression for these boundaries
and their limit in a linear framework can be found
in~\cite{ColomboGuerraInc}.

\section{Technical Details}
\label{sec:TD}

Throughout, we suppose that $P, P_g$ in theorems~\ref{thm:kappa},
\ref{thm:limit} satisfy condition \textbf{(P)} and denote by $\O$ a
quantity that depends only on $P, P_g$ and on uniform bounds on the
initial data.

We define $\mathcal{T}, \mathcal{T}_\kappa, \mathcal{T}_g$ as
in~\eqref{eq:Tproperties}, \eqref{eq:tkfamily} and collect below a few
facts about the $p$-system in Lagrangian coordinates using the $(p,v)$
plane. Consider first the gas phase, where
\begin{equation}
  \label{eq:lambdaGas}
  \left\{
    \begin{array}{ll}
      \partial_t \mathcal{T}_{g}(p) - \partial_z v = 0
      \\
      \partial_t v + \partial_z p = 0 \,,
    \end{array}
  \right.
  \quad \mbox{ with eigenvalues } \quad
  \begin{array}{rcl}
    \lambda_1^g(p,v)
    & = &
    -\sqrt{-\dfrac{1}{\mathcal{T}'_{g}(p)}}
    \\
    \lambda_2^g(p,v)
    & = &
    \sqrt{-\dfrac{1}{\mathcal{T}'_{g}(p)}}
  \end{array}
\end{equation}
so that the Lax shock and rarefaction curves are, see
also~\cite{ColomboGuerraSchleper},
\begin{equation}
  \label{eq:LaxGas}
  \begin{array}{rcl}
    V_1^g (p; p_o,v_o)
    & = &
    \left\{
      \begin{array}{l@{\qquad}r@{\;}c@{\;}l}
        \displaystyle
        v_o - \int_{p_o}^p \sqrt{-\mathcal{T}'_g (\xi)} \d{\xi}
        & p & < & p_o
        \\
        \displaystyle
        v_o
        -
        \sqrt{-\left(\mathcal{T}_g (p) -\mathcal{T}_g (p_o)\right) \left(p-p_o\right)}
        & p & \geq & p_o
      \end{array}
    \right.
    \\
    V_2^g (p; p_o,v_o)
    & = &
    \left\{
      \begin{array}{l@{\qquad}r@{\;}c@{\;}l}
        \displaystyle
        v_o
        -
        \sqrt{-\left(\mathcal{T}_g (p) -\mathcal{T}_g (p_o)\right) \left(p-p_o\right)}
        & p & < & p_o
        \\
        \displaystyle
        v_o + \int_{p_o}^p \sqrt{-\mathcal{T}'_g (\xi)} \d{\xi}
        & p & \geq & p_o \,.
      \end{array}
    \right.
  \end{array}
\end{equation}
Similarly, in the liquid phase we have
\begin{equation}
  \label{eq:lambdaLiquid}
  \left\{
    \begin{array}{ll}
      \partial_t \mathcal{T}_{\kappa}(p) - \partial_z v = 0
      \\
      \partial_t v + \partial_z p = 0 \,.
    \end{array}
  \right.
  \quad \mbox{ with eigenvalues } \quad
  \begin{array}{rcl}
    \lambda_1^\kappa(p,v)
    & = &
    -\dfrac{1}{\kappa}
    \sqrt{-\dfrac{1}{\mathcal{T}'\left(\bar p + \kappa^{2}\left(p-\bar p\right)\right)}}
    \\
    \lambda_2^\kappa(p,v)
    & = &
    \dfrac{1}{\kappa}
    \sqrt{-\dfrac{1}{\mathcal{T}'\left(\bar p + \kappa^{2}\left(p-\bar p\right)\right)}}
  \end{array}
\end{equation}
and the Lax curves are
\begin{equation}
  \label{eq:LaxLiquid}
  \begin{array}{rcl}
    V_1^\kappa (p; p_o,v_o)
    & = &
    \left\{
      \begin{array}{l@{\qquad}r@{\;}c@{\;}l}
        \displaystyle
        v_o - \int_{p_o}^p \sqrt{-\mathcal{T}'_\kappa (\xi)} \d{\xi}
        & p & < & p_o
        \\
        \displaystyle
        v_o
        -
        \sqrt{-\left(\mathcal{T}_\kappa (p) -\mathcal{T}_\kappa (p_o)\right) \left(p-p_o\right)}
        & p & \geq & p_o
      \end{array}
    \right.
    \\
    V_2^\kappa (p; p_o,v_o)
    & = &
    \left\{
      \begin{array}{l@{\qquad}r@{\;}c@{\;}l}
        \displaystyle
        v_o
        -
        \sqrt{-\left(\mathcal{T}_\kappa (p) -\mathcal{T}_\kappa (p_o)\right) \left(p-p_o\right)}
        & p & < & p_o
        \\
        \displaystyle
        v_o + \int_{p_o}^p \sqrt{-\mathcal{T}'_\kappa (\xi)} \d{\xi}
        & p & \geq & p_o \,.
      \end{array}
    \right.
  \end{array}
\end{equation}

Below we systematically use the parameterizations
\begin{equation}
  \label{eq:para_sigma}
  \sigma_i \to V_i^g (p_o+\sigma_i; p_o,v_o)
  \qquad \mbox{ and } \qquad
  \sigma_i \to V_i^\kappa (p_o+\sigma_i;p_o,v_o)
\end{equation}
of the $i$--Lax curve, $\sigma_i$ being a pressure
difference. Therefore, differently from the usual habit, we have that
\begin{table}[!h]
  \centering
  \begin{tabular}{|l||l|}
    \hline
    \hfil $i=1$
    &
    \hfil$i=2$
    \\\hline
    $\sigma_1 < 0$ $\Rightarrow$ \mbox{rarefaction}
    &
    $\sigma_2 < 0$ $\Rightarrow$ \mbox{shock}
    \\\hline
    $\sigma_1 > 0$ $\Rightarrow$ \mbox{shock}
    &
    $\sigma_2 > 0$ $\Rightarrow$ \mbox{rarefaction}
    \\\hline
  \end{tabular}
  \caption{Types of waves and the signs of the corresponding parameters as in~\eqref{eq:LaxGas}, \eqref{eq:LaxLiquid}.}
  \label{tab:signs}
\end{table}

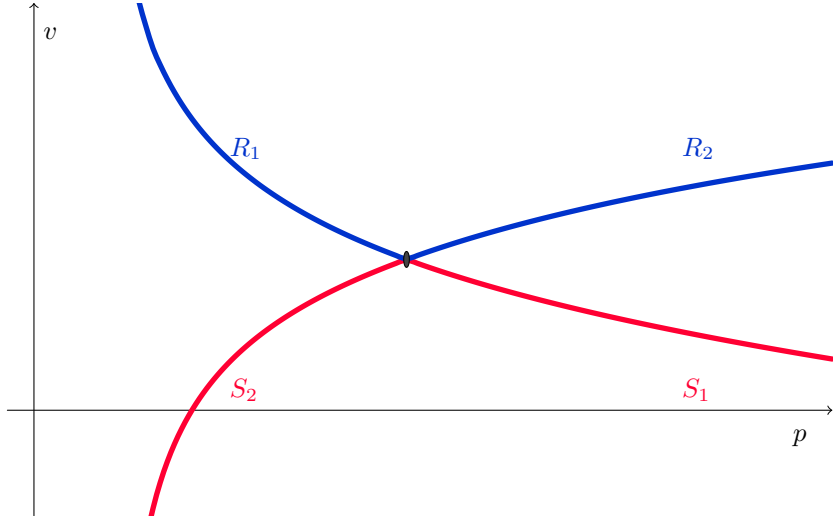
\begin{figure}[!h]
  \centering \definecolor{qqttcc}{rgb}{-0.5,0.2,0.8}
  \definecolor{ffqqtt}{rgb}{1.0,0.0,0.2}
  \definecolor{uuuuuu}{rgb}{0.26666666666666666,0.26666666666666666,0.26666666666666666}
  \begin{tikzpicture}[line cap=round,line
    join=round,yscale=2,xscale=0.7]
    \draw[->,color=black] (-0.5,0.0) -- (15.0,0.0); \clip(-0.5,-0.7)
    rectangle (15.0,2.7); \draw[line
    width=2.0pt,color=ffqqtt,smooth,samples=100,domain=1.7:15.0]
    plot(\x,{1.0-(1.4*abs(7.0-(\x)))/sqrt((3.0+1.4^(2.0)*((\x)-3.0))*(3.0+1.4^(2.0)*(7.0-3.0)))});
    \draw[line
    width=2.0pt,color=qqttcc,smooth,samples=100,domain=1.7:15.0]
    plot(\x,{1.0+1.0/1.4*abs(ln(1.0+(1.4^(2.0)*((\x)-7.0))/(3.0+1.4^(2.0)*(7.0-3.0))))});
    \begin{scriptsize}
      \draw [fill=uuuuuu] (7.0,1.0) circle (1.5pt);
    \end{scriptsize}
    \draw[->,color=black] (0.0,-0.7) -- (0.0,2.7); \draw[color=black]
    (14.0843414275202356,-0.3) node [anchor=south west] {$p$};
    \draw[color=black] (0,2.4) node [anchor=south west] {$v$};
    \draw[color=qqttcc] (3.5,1.6) node [anchor=south west] {$R_{1}$};
    \draw[color=qqttcc] (12.0,1.6) node [anchor=south west] {$R_{2}$};
    \draw[color=ffqqtt] (3.5,0.0) node [anchor=south west] {$S_{2}$};
    \draw[color=ffqqtt] (12.0,0.0) node [anchor=south west] {$S_{1}$};
  \end{tikzpicture}
  \caption{Lax curves~\eqref{eq:LaxLiquid} in the $(p,v)$--plane.}
  \label{fig:LaxCurves}
\end{figure}

\begin{lemma}
  \label{lem:para}
  Fix $L,l$ with $L>l>0$ and $\kappa \in \left]0, 1 \right]$. The Lax
  curves~\eqref{eq:LaxLiquid} admit the representation
  \begin{equation}
    \label{eq:para}
    \begin{array}{rcl}
      V_1^\kappa (p;p_o,v_o)
      & = &
      v_o - \kappa \, (p-p_o) \, F\left(\Pi_\kappa (p), \Pi_\kappa (p_o)\right)
      \\[6pt]
      V_2^\kappa (p;p_o,v_o)
      & = &
      v_o + \kappa \, (p-p_o) \, F\left(\Pi_\kappa (p_o), \Pi_\kappa (p)\right)
    \end{array}
  \end{equation}
  where
  \begin{eqnarray}
    \nonumber
    \Pi_\kappa (p) & = & \bar p + \kappa^2 \, (p-\bar p) \,,
    \\
    \label{eq:F}
    F (x,y) & = &
    \left\{
      \begin{array}{l@{\qquad}r@{\;}c@{\;}l}
        \displaystyle
        \int_0^1
        \sqrt{-\mathcal{T}'\left(\theta x + (1-\theta)y\right)}
        \d\theta
        & x & < & y \,,
        \\[12pt]
        \displaystyle
        \sqrt{-\mathcal{T}'\left(x\right)}
        & x & = & y \,,
        \\[12pt]
        \displaystyle
        \sqrt{\int_0^1 -\mathcal{T}'\left(\theta x + (1-\theta)y\right) \d\theta}
        & x & > & y \,.
      \end{array}
    \right.
  \end{eqnarray}
  Moreover,
  \begin{enumerate}
  \item \label{it:C1} the function $F$ is of class $\C{1,1}
    ([l,L]^2;\reali)$;
  \item \label{it:Ctanto} both restrictions $F_{\strut\vert x\leq y}$
    and $F_{\strut\vert x \geq y}$ are of class $\C2([l,L]^2;\reali)$;
  \item \label{it:Bound}for $x,y \in [l,L]$, $F (x,y) \in
    \left[\sqrt{-\mathcal{T}' (L)}, \sqrt{-\mathcal{T}' (l)}\right]$.
  \end{enumerate}
\end{lemma}

\noindent The proof follows from standard computations. A property
that plays a key role in the sequel is that the function $F$ above is
independent of $\kappa$.

Call $F_g$ the function obtained Replacing $\mathcal{T}$ with
$\mathcal{T}_g$ in~\eqref{eq:F}. Then, Lemma~\ref{lem:para} in the
case $\kappa = 1$, yields a representation for the Lax
curve~\eqref{eq:LaxGas} in the gas phase.

\paragraph{Riemann Solvers.} The wave front tracking algorithm below
is, as usual, based on the (possibly, approximate) solutions to
Riemann problems.

Throughout, we fix a reference pressure $p_o>0$. By Galileian
invariance, in the statements below only speed differences will be
relevant.

\begin{lemma}
  \label{lem:RS_liquid}
  There exists a positive $\bar\delta$ such that for all $\kappa \in
  \left]0, 1\right]$ and for any couple of states $(p^l, v^l)$, $(p^r,
  v^r)$ with $\modulo{p^l-p_o} + \modulo{p^r-p_o} < \bar \delta$ and
  $\modulo{v^r-v^l} < \kappa \, \bar \delta$, there exists a unique
  state $(p^m, v^m)$ satisfying
  \begin{displaymath}
    V_1^\kappa (p^m;p^l,v^l) = v^m
    \qquad \mbox{ and } \qquad
    V_2^\kappa (p^r; p^m, v^m) = v^r \,.
  \end{displaymath}
  Moreover,
  \begin{equation}
    \label{eq:es1}
    \modulo{p^l - p^m}
    +
    \modulo{p^m - p^r}
    \leq
    \O
    \left(
      \modulo{p^l - p^r}
      +
      \frac{\modulo{v^l - v^r}}{\kappa}
    \right) .
  \end{equation}
\end{lemma}

A qualitative justification of~\eqref{eq:es1} is provided in
Figure~\ref{fig:estriemann}.  In the liquid region, the Lax curves
have a slope of order $\kappa$ (see Lemma~\ref{lem:para}), hence a
jump $\Delta v$ in the velocity generates waves of order $\Delta
v/\kappa$.
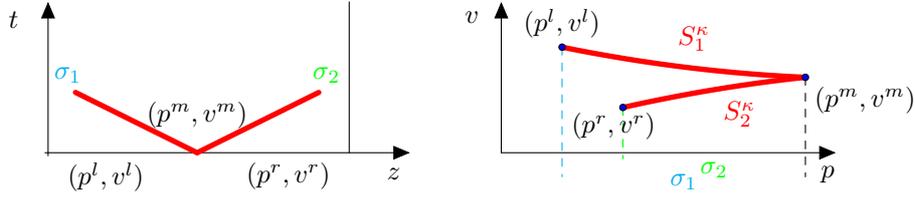
\begin{figure}[!ht]
  \centering
  \begin{tikzpicture}[line cap=round,line join=round,>=triangle
    45,xscale=0.4,yscale=0.4]
    \draw [->] (-4.0,-1.0) -- (8.0,-1.0); \draw [->] (-3.9,-1.1) --
    (-3.9,4.0); \draw (6.0,-1.0) -- (6.0,4.0); \draw (8,-1.2)
    node[anchor=north east] {$z$}; \draw (-5.5,4) node[anchor=north
    west] {$t$}; \draw (-2,-1) node[anchor=north] {$(p^l,v^l)$}; \draw
    (4,-1) node[anchor=north] {$(p^r,v^r)$}; \draw (1,1)
    node[anchor=north] {$(p^m,v^m)$}; \draw (21,1.5) node[anchor=north
    west] {$(p^m,v^m)$}; \draw [line width=2pt,color=red] (1,-1) --
    (-3,1); \draw [line width=2pt,color=red] (1,-1) -- (5,1); \draw
    [->] (11.0,-1.0) -- (22.0,-1.0); \draw [->] (11.0,-1.0) --
    (11.0,4.0); \draw (9.5,4) node[anchor=north west] {$v$}; \draw
    [line width=2.0pt,color=red] plot[domain=13:21,variable=\t]
    ({\t},{2.5 - (\t-13)/8 + 0.01 * (\t * \t - 34 * \t + 273)}); \draw
    [line width=2.0pt,color=red] plot[domain=15:21,variable=\t]
    ({\t},{0.5+(\t-15.)/ 6. - 0.01 * (\t * \t - 36 * \t + 315)});
    \draw [dashed] (21,1.5) -- (21,-1.8); \draw [dashed,color=green]
    (15,0.5) -- (15,-1.2); \draw [dashed,color=cyan] (13,2.5) --
    (13,-1.8); \draw (18,-1) node[anchor=north,color=green]
    {$\sigma_{2}$}; \draw (17,-1.4) node[anchor=north,color=cyan]
    {$\sigma_{1}$}; \draw (4.5,1) node[anchor=south west,color=green]
    {$\sigma_{2}$}; \draw (-4.,1) node[anchor=south west,color=cyan]
    {$\sigma_{1}$}; \draw [color=red](16.5,3.5) node[anchor=north
    west] {$S_1^{\kappa}$}; \draw [color=red](18.,1.)
    node[anchor=north west] {$S_2^{\kappa}$}; \draw
    (21.19999999999993,-1.1238532110091715) node[anchor=north west]
    {$p$}; \draw (13.0,0.6307339449541268) node[anchor=north west]
    {$(p^r,v^r)$}; \draw (13.0,2.5) node[anchor=south] {$(p^l,v^l)$};
    \begin{scriptsize}
      \draw [fill=blue] (13,2.5) circle (3pt); \draw [fill=blue]
      (15,0.5) circle (3pt); \draw [fill=blue] (21,1.5) circle (3pt);
    \end{scriptsize}
  \end{tikzpicture}
  \caption{Riemann problem in the liquid.  Left, in the $(t,z)$ plane
    and, right, in the $(p,v)$ plane:
    $\left|\sigma_{1}\right|+\left|\sigma_{2}\right|=\O \left( |p^l -
      p^r| + |v^l - v^r|/\kappa \right) .  $}
  \label{fig:estriemann}
\end{figure}

\begin{proof}
  Let $\xi = (v^r - v^l)/\kappa$.  We apply the Implicit Function
  Theorem to $G (p^m,p^l,p^r,\xi) = 0$ where
  \begin{displaymath}
    G (p^m,p^l,p^r,\xi) =
    \xi
    +
    (p^m-p^l) \,
    F\left(\Pi_\kappa (p^m), \Pi_\kappa (p^l)\right)
    -
    (p^r-p^m) \, F\left(\Pi_\kappa (p^m),\Pi_\kappa (p^r)\right)
  \end{displaymath}
  to find $p^m$ as a function of $(p^l,p^r,\xi)$, which is possible
  since the derivative $\partial_{p^m} G$ evaluated at $p^{m}=p^l =
  p^r = p_o$ and $\xi = 0$ is
  \begin{displaymath}
    \begin{split}
      \partial_{p^m} G (p_o,p_o,p_o,0) &= 2
      F\left(\Pi_{\kappa}\left(p_{o}\right),\Pi_{\kappa}\left(p_{o}\right)\right)\\
      &=2\sqrt{-\mathcal{T}'\left(\bar p + \kappa^{2}\left(p_{o}-\bar
            p\right)\right)}\ge2\sqrt{-\mathcal{T}'\left(\max\left\{\bar
            p,p_{o}\right\}\right)}>0
    \end{split}
  \end{displaymath}
  Note also that, in a neighborhood of
  $\left(p_{o},p_{o},p_{o},0\right)$, all second derivatives of $G$
  are bounded uniformly in $\kappa$, hence the domain of the implicit
  function contains a neighborhood of $(p_o,p_o,p_o,0)$ independent of
  $\kappa$. Finally, $v^m$ can be computed as $v^m = v^l - \kappa \,
  (p^m-p^l) \, F\left(\Pi_\kappa (p^m), \Pi_\kappa (p^l)\right)$, by
  Lemma~\ref{lem:para}.

  Finally, \eqref{eq:es1} follows from
  $G\left(p^{l},p^{l},p^{l},0\right)=G\left(p^{r},p^{r},p^{r},0\right)=0$
  and the Lipschitz continuity of the implicit function.
\end{proof}

Note that Lemma~\ref{lem:RS_liquid} in the case $\kappa = 1$ covers
the case of Riemann problems in the gas phase and slightly
improves~\cite[Chapter~5]{BressanLectureNotes}.

The next Lemma refers to the Riemann problem between the gas, on the
left, and the liquid, on the right. The symmetric situation is
entirely similar.

\begin{lemma}
  \label{lem:RS_interface}
  There exits a positive $\bar\delta$ such that for all $\kappa \in
  \left]0, 1\right]$ and for any couple of states $(p^l, v^l)$, $(p^r,
  v^r)$ with $\modulo{p^l-p_o} + \modulo{p^r-p_o} < \bar\delta$ and
  $\modulo{v^r-v^l} < \bar \delta$, there exists a unique state $(p^m,
  v^m)$ satisfying
  \begin{displaymath}
    V_1^g (p^m;p^l,v^l) = v^m
    \qquad \mbox{ and } \qquad
    V_2^\kappa (p^r; p^m, v^m) = v^r \,.
  \end{displaymath}
  Moreover,
  \begin{eqnarray}
    \label{eq:inu1}
    \modulo{p^m - p^l}
    & = &
    \O
    \left(
      \kappa \, \modulo{p^r-p^l}
      +
      \modulo{v^l - v^r}
    \right)
    \\
    \label{eq:inu2}
    \modulo{p^m - p^r}
    & = &
    \O
    \left(
      \modulo{p^r-p^l}
      +
      \modulo{v^l - v^r}
    \right)
    \\
    \label{eq:inu3}
    \frac{1}{\kappa} \,
    \modulo{v^m - v^r}
    & = &
    \O
    \left(
      \modulo{p^r-p^l}
      +
      \modulo{v^l - v^r}
    \right)
  \end{eqnarray}
\end{lemma}

\begin{proof}
  Let $\xi = v^r - v^l$. We apply the Implicit Function Theorem to $G
  (p^m,p^l,p^r,\xi) = 0$ where
  \begin{displaymath}
    G (p^m,p^l,p^r,\xi) =
    \xi
    +
    (p^m-p^l) \, F_g(p^m, p^l)
    -
    \kappa \, (p^r-p^m) \, F\left(\Pi_\kappa (p^m),\Pi_\kappa (p^r)\right)
  \end{displaymath}
  to find $p^m$ as a function of $(p^l,p^r,\xi)$, which is possible
  since the derivative $\partial_{p^m} G$ evaluated at $p^{m}=p^l =
  p^r = p_o$ and $\xi = 0$ is
  \begin{displaymath}
    \partial_{p^m} G (p_o,p_o,p_o,0)
    =
    \sqrt{-\mathcal{T}'_g (p_o)}
    +
    \kappa \,
    F\left(\Pi_{\kappa}\left(p_{o}\right),\Pi_{\kappa}\left(p_{o}\right)\right)\ge
    \sqrt{-\mathcal{T}'_g (p_o)}>0\,.
  \end{displaymath}
  Note also that, in a neighborhood of $\left(p_o,p_o,p_o,0\right)$,
  all second derivatives of $G$ are bounded uniformly in $\kappa$,
  hence the domain of the implicit function contains a neighborhood of
  $(p_o,p_o,p_o,0)$ independent of $\kappa$. Moreover, $v^m$ can be
  computed as $v^m = v^l - (p^m-p^l) \, F_g (p^m, p^l)$, by
  Lemma~\ref{lem:para}. Concerning the latter estimates, use $G
  (p^m,p^l, p^r, v^r-v^l) = 0$ to obtain
  \begin{eqnarray*}
    p^m - p^l
    & = &
    \frac{
      v^l - v^r
      +
      \kappa \, (p^r-p^l)F\left(\Pi_\kappa (p^m), \Pi_\kappa (p^r)\right)
    }{F_g (p^m,p^l) + \kappa \, F\left(\Pi_\kappa (p^m), \Pi_\kappa (p^r)\right)}
  \end{eqnarray*}
  which implies~\eqref{eq:inu1} and, together with the simple
  inequality $\modulo{p^m - p^r} \leq \modulo{p^m-p^l} + \modulo{p^l -
    p^r}$ also proves~\eqref{eq:inu2}. Finally, the equality $v^r -
  v^m = \kappa \, (p^r - p^m) \, F\left(\Pi_\kappa (p^m), \Pi_\kappa
    (p^r)\right)$, together with~\eqref{eq:inu2},
  proves~\eqref{eq:inu3}.
\end{proof}

\paragraph{Definition of the Algorithm.}

We modify the standard construction of the wave front tracking
algorithm, see for instance~\cite[Chapter~4]{BressanLectureNotes}.

First, we identify the state $u$ by means of the pair $(p,v)$. Indeed,
we choose to parametrize the Lax curves as
in~\eqref{eq:LaxGas}--\eqref{eq:LaxLiquid} and, hence, the waves'
sizes are measured through the pressure difference $\sigma$ between
the two states on the sides of the wave.

Second, we introduce two strips around the two interfaces $z = 0$ and
$z = m$, where all $1$-waves have speed $-1$ and all $2$-waves have
speed $1$. This, together with~\cite[Lemma~2.5]{relaxation}, allows to
avoid the introduction of non-physical waves, significantly
simplifying the whole procedure.

We consider a representative of the initial datum $\tilde
u\in\left(\BV\cap\L1\right)\left(\reali,\reali^{+}\times\reali\right)$
such that $\tilde u(0+)=\tilde u(0)$, $\tilde u(m-)=\tilde u(m)$.  Fix
$\epsilon > 0$. We approximate the initial datum $\tilde u$ by a
sequence $\tilde u^\epsilon$ of piecewise constant initial data with a
finite number of discontinuities such that:
\begin{equation}
  \label{eq:hyp_id}
  \begin{array}{@{}rcl@{\quad}rclr@{\;}c@{\;}l@{}}
    \tv (\tilde p^\epsilon) & \leq & \tv (\tilde p) \,,
    &
    \norma{\tilde u^\epsilon - \tilde u}_{\L1} & \leq & \epsilon \,,
    \\
    \tv (\tilde v^\epsilon;\mathcal{G}) & \leq & \tv (\tilde v;\mathcal{G}) \,,
    &
    \tilde u^\epsilon (z) & = & \tilde u (0)
    & \mbox{ for all }
    z & \in & [-2\epsilon^2, 2\epsilon^2] \,,
    \\
    \tv (\tilde v^\epsilon;\mathcal{L}) & \leq & \tv (\tilde v;\mathcal{L})  \,,
    &
    u^o_\epsilon (z) & = & u^o (m)
    & \mbox{ for all }
    z & \in & [m-2\epsilon^2, m+2\epsilon^2]\,.
  \end{array}
\end{equation}
Observe that a possible jump at the interfaces $z=0$ and $z=m$ is
assigned to the gas region.  At each point of jump in the approximate
initial datum, we solve the corresponding Riemann problem. As usual,
see~\cite[Chapter~4]{BressanLectureNotes}, we approximate each
rarefaction wave by a rarefaction fan consisting of
$\epsilon$-wavelets, each with strength less than $\varepsilon$ and
traveling with the characteristic speed of the state to its left. On
the other hand, each shock wave is assigned its exact Rankine-Hugoniot
speed.  Similarly to what happens in the usual case, there exists a
constant $\delta_{o} > 0$ such that each of the above Riemann problems
has an approximate solution as long as $\tv (\tilde u) < \delta_{o}$.
We introduce two strips around the two interfaces $z = 0$ and $z = m$,
where all $1$-waves have speed $-1$ and all $2$-waves have speed $+1$:
\begin{displaymath}
  \mathcal{I}_\epsilon^- = [-\epsilon^2, \epsilon^2] \times \reali^+
  \quad \mbox{ and } \quad
  \mathcal{I}_\epsilon^+ = [m-\epsilon^2, m+\epsilon^2] \times \reali^+ \,.
\end{displaymath}
This, together with~\cite[Lemma~2.5]{relaxation}, allows us to avoid
the introduction of non-physical waves, significantly simplifying the
whole procedure.  Hence, assign to all $1$-waves entering
$\mathcal{I}_\epsilon^- \cup \mathcal{I}_\epsilon^+$ speed $-1$, while
all $2$-waves entering $\mathcal{I}_\epsilon^- \cup
\mathcal{I}_\epsilon^+$ are given speed $+1$, see
Figure~\ref{fig:regions}.

Remark that the actual values attained by the approximate solution are
not changed, only the wave speeds are modified. When exiting these
strips, every wave is given back its correct speed.  By this trick, no
interaction among waves of the same family may take place in either of
the two strips.  This construction can be extended up to the first
time $t_1$ at which two waves interact, or a wave hits one of the
interfaces.  At time $t_1$, the so constructed approximate solution is
piecewise constant with a finite number of discontinuities.  Any such
interaction gives rise to a new Riemann problem solved as at time
$t=0$, if the interaction is in the interior of the two phases, or as
described in Lemma~\ref{lem:RS_interface}, whenever the interaction is
along an interface.

Any rarefaction wave, once arisen, is not further split even if its
strength exceeds the threshold $\varepsilon$ after subsequent
interactions, with other waves or with the phase boundaries. The new
rarefaction waves that may arise at the interfaces are split, if their
strength exceeds $\varepsilon$, when they exit the strips
$\mathcal{I}_\epsilon^{\pm}$, since inside the strips they all travel
with the same speed.  We can thus iterate the previous construction at
any subsequent interaction, provided suitable upper bounds on the
total variation of the approximate solutions are available.  As it is
usual in this context, see~\cite[Chapter~7]{BressanLectureNotes}, we
may assume that no more than 2 waves interact at any interaction
point, or that no interaction happens at the boundaries of the two
strips, thanks to a small modification of the speed of waves outside
the strips, where necessary.

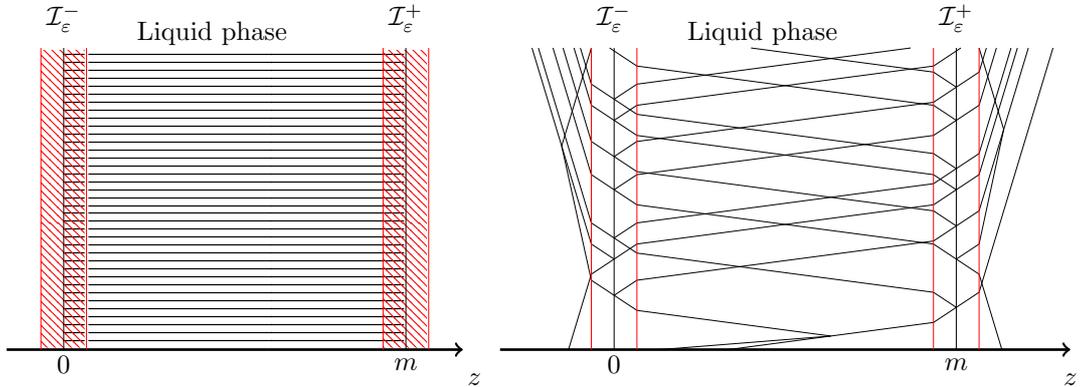
\begin{figure}[!h]
  \centering
  \begin{tikzpicture}[line width=0.5mm,yscale=2, xscale=3]
    \draw[color=white,pattern=horizontal lines] (0,-1) -- (1.5,-1) --
    (1.5,1) -- (0,1); \draw[color=white,pattern=north west lines,
    pattern color=red] (-0.1,-1) -- (0.1,-1) -- (0.1,1) -- (-0.1,1);
    \draw[color=white,pattern=north west lines, pattern color=red]
    (1.4,-1) -- (1.6,-1) -- (1.6,1) -- (1.4,1); \node at (0,1.2)
    {$\mathcal{I}_{\epsilon}^-$}; \node at (1.5,1.2)
    {$\mathcal{I}_{\epsilon}^+$}; \node at (0,-1.1) {$0$}; \node at
    (1.5,-1.1) {$m$}; \node at (0.65,1.1) {Liquid phase}; \draw[line
    width=0.2] (0,-1) -- (0,1); \draw[line width=1,->] (-0.25,-1) --
    (1.75,-1); \draw[line width=0.2] (1.5,-1) -- (1.5,1); \draw[line
    width=0.1,color=red] (-0.1,-1) -- (-0.1,1); \draw[line
    width=0.1,color=red] (0.1,-1) -- (0.1,1); \draw[line
    width=0.1,color=red] (1.4,-1) -- (1.4,1); \draw[line
    width=0.2,color=red] (1.6,-1) -- (1.6,1); \node at (1.8,-1.2)
    {$z$};
  \end{tikzpicture}
  \begin{tikzpicture}[line width=0.5mm,yscale=2, xscale=3]
    \node at (0,1.2) {$\mathcal{I}_{\epsilon}^-$}; \node at (1.5,1.2)
    {$\mathcal{I}_{\epsilon}^+$}; \node at (0,-1.1) {$0$}; \node at
    (1.5,-1.1) {$m$}; \node at (0.65,1.1) {Liquid phase}; \draw[line
    width=0.2] (-0.200,-1.000) -- (-0.100,-0.500) -- (0.000,-0.400) --
    (0.100,-0.300) -- (1.400,-0.040) -- (1.500,0.060) -- (1.600,0.160)
    -- (1.768,1.000); \draw[line width=0.2] (-0.000,-0.400) --
    (-0.100,-0.300) -- (-0.360,1.000); \draw[line width=0.2]
    (1.500,0.060) -- (1.400,0.160) -- (0.100,0.420) -- (0.000,0.520)
    -- (-0.100,0.620) -- (-0.176,1.000); \draw[line width=0.2]
    (0.500,-1.000) -- (1.400,-0.820) -- (1.500,-0.720) --
    (1.600,-0.620) -- (1.924,1.000); \draw[line width=0.2] (0.2,-1) --
    (0.95,-0.91); \draw[line width=0.2] (0.950,-0.910) --
    (0.100,-0.740) -- (0.000,-0.640) -- (-0.100,-0.540) --
    (-0.23,0.35); \draw[line width=0.2] (-0.230,0.350) --
    (-0.100,1.000); \draw[line width=0.2] (1.700,-1.000) --
    (1.600,-0.500) -- (1.500,-0.400) -- (1.400,-0.300) --
    (0.100,-0.040) -- (0.000,0.060) -- (-0.100,0.160) --
    (-0.268,1.000); \draw[line width=0.2] (0.000,0.520) --
    (0.100,0.620) -- (1.400,0.880) -- (1.500,0.980) -- (1.520,1.000);
    \draw[line width=0.2] (1.500,-0.720) -- (1.400,-0.620) --
    (0.100,-0.360) -- (0.000,-0.260) -- (-0.100,-0.160) --
    (-0.332,1.000); \draw[line width=0.2] (0.000,-0.640) --
    (0.100,-0.540) -- (1.400,-0.280) -- (1.500,-0.180) --
    (1.600,-0.080) -- (1.816,1.000); \draw[line width=0.2]
    (1.708,0.460) -- (1.600,1.000) -- (1.600,1.000); \draw[line
    width=0.2] (1.500,-0.400) -- (1.600,-0.300) -- (1.708,0.46);
    \draw[line width=0.2] (0.000,0.060) -- (0.100,0.160) --
    (1.400,0.420) -- (1.500,0.520) -- (1.600,0.620) -- (1.676,1.000);
    \draw[line width=0.2] (0.000,-0.260) -- (0.100,-0.160) --
    (1.400,0.100) -- (1.500,0.200) -- (1.600,0.300) -- (1.740,1.000);
    \draw[line width=0.2] (1.500,-0.180) -- (1.400,-0.080) --
    (0.100,0.180) -- (0.000,0.280) -- (-0.100,0.380) --
    (-0.224,1.000); \draw[line width=0.2] (1.500,0.200) --
    (1.400,0.300) -- (0.100,0.560) -- (0.000,0.660) -- (-0.100,0.760)
    -- (-0.148,1.000); \draw[line width=0.2] (1.500,0.520) --
    (1.400,0.620) -- (0.100,0.880) -- (0.000,0.980) -- (-0.020,1.000);
    \draw[line width=0.2] (0.000,0.280) -- (0.100,0.380) --
    (1.400,0.640) -- (1.500,0.740) -- (1.600,0.840) -- (1.632,1.000);
    \draw[line width=0.2] (0.000,0.660) -- (0.100,0.760) --
    (1.300,1.000); \draw[line width=0.2] (1.500,0.740) --
    (1.400,0.840) -- (0.600,1.000); \draw[line width=0.2] (0,-1) --
    (0,1); \draw[line width=1,->] (-0.50,-1) -- (2,-1); \draw[line
    width=0.2] (1.5,-1) -- (1.5,1); \draw[line width=0.1,color=red]
    (-0.1,-1) -- (-0.1,1); \draw[line width=0.1,color=red] (0.1,-1) --
    (0.1,1); \draw[line width=0.1,color=red] (1.4,-1) -- (1.4,1);
    \draw[line width=0.2,color=red] (1.6,-1) -- (1.6,1); \node at
    (2.0,-1.2) {$z$};
  \end{tikzpicture}
  \caption{Left, the strips $\mathcal{I}_\epsilon^{\pm}$ and the
    liquid phase. Right, modification to the usual wave front tracking
    algorithm: waves in the strips $\mathcal{I}_\epsilon^-$ and
    $\mathcal{I}_\epsilon^+$ are assign speed $1$, if belonging to the
    first family, and $-1$, if of the second family.}
  \label{fig:regions}
\end{figure}

\paragraph{Interaction Estimates.} We recall the classical Glimm
interaction estimates, see~\cite[Chapter~7,
formul\ae~(7.31)--(7.32)]{BressanLectureNotes}, which hold for any
smooth parametrization of the Lax curves:
\begin{figure}[!h]
  \centering
  \begin{tikzpicture}[line width=0.5mm,yscale=1.5, xscale=1.5]
    \node at (-1.1,-1.1) {$\sigma_{2}^{-}$}; \node at (1.0,-1.1)
    {$\sigma_{1}^{-}$}; \node at (1.1,1.1) {$\sigma_{2}^{+}$}; \node
    at (-1.1,1.1) {$\sigma_{1}^{+}$}; \node at (-1,0)
    {$\left(p_{o},v_{o}\right)$}; \draw[-] (1,1) -- (0,0); \draw[-]
    (-1,1) -- (0,0); \draw[-] (-1.2,-1) -- (0,0); \draw[-] (0.8,-1) --
    (0,0);
  \end{tikzpicture}
  \begin{tikzpicture}[line width=0.5mm,yscale=1.5, xscale=1.5]
    \node at (-1.1,-1.1) {$\sigma'$}; \node at (-0.5,-1.1)
    {$\sigma''$}; \node at (1.1,1.1) {$\sigma_{2}^{+}$}; \node at
    (-1.1,1.1) {$\sigma_{1}^{+}$}; \node at (-1,0)
    {$\left(p_{o},v_{o}\right)$}; \draw[-] (-1,-1) -- (1,1); \draw[-]
    (0,0) -- (-1,1); \draw[-] (-0.5,-1) -- (0,0);
  \end{tikzpicture}
  \begin{tikzpicture}[line width=0.5mm,yscale=1.5, xscale=1.5]
    \node at (0.5,-1.1) {$\sigma'$}; \node at (1.1,-1.1) {$\sigma''$};
    \node at (1.1,1.1) {$\sigma_{2}^{+}$}; \node at (-1.1,1.1)
    {$\sigma_{1}^{+}$}; \node at (-1,0) {$\left(p_{o},v_{o}\right)$};
    \draw[-] (0,0) -- (1,1); \draw[-] (1,-1) -- (-1,1); \draw[-]
    (0.5,-1) -- (0,0);
  \end{tikzpicture}
  \caption{Left, an interaction between waves of different
    families. Center, an interaction between waves of the second
    family. Right, an interaction between waves of the first
    family.
    \label{fig:interactions}}
\end{figure}
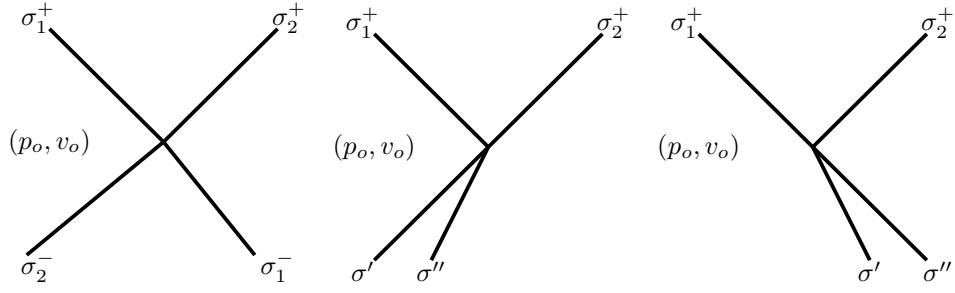
\begin{equation}
  \label{eq:ieStandard}
  \begin{array}{rcl@{\qquad\qquad}l}
    \modulo{\sigma_1^+ - \sigma_1^-}
    +
    \modulo{\sigma_2^+ - \sigma_2^-}
    & \leq &
    \O \modulo{\sigma_1^- \sigma_2^-}
    & \mbox{(Figure~\ref{fig:interactions}, left),}
    \\
    \modulo{\sigma_1^+}
    +
    \modulo{\sigma_2^+ - (\sigma'+\sigma'')}
    & \leq &
    \O \modulo{\sigma'\sigma''}
    & \mbox{(Figure~\ref{fig:interactions}, middle),}
    \\
    \modulo{\sigma_1^+ - (\sigma'+\sigma'')}
    +
    \modulo{\sigma_2^+}
    & \leq &
    \O \modulo{\sigma'\sigma''}
    & \mbox{(Figure~\ref{fig:interactions}, right),}
  \end{array}
\end{equation}
where we used the notation described in Figure~\ref{fig:interactions}.

Aiming at the convergence result, we need more careful interaction
estimates in the liquid phase. More precisely, we seek bounds on the
constant $\O$ above that allow to control its dependence on
$\kappa$. Remark that the choice of parametrizing Lax curves by means
of pressure differences plays a key role in this improvement.

\begin{lemma}
  \label{lem:InterLiquid}
  There exists a $\bar \delta >0$ such that if the interacting waves
  in Figure~\ref{fig:interactions} hit each other in $\mathcal{L}$ and
  all have sizes less than $\bar\delta$, then, the following estimates
  hold:
  \begin{equation}
    \label{eq:ieNonStandard}
    \begin{array}{rcl}
      \modulo{\sigma_1^+ - \sigma_1^-}
      +
      \modulo{\sigma_2^+ - \sigma_2^-}
      & \leq &
      \O \, \kappa^2 \, \modulo{\sigma_1^- \sigma_2^-}
      \\
      \modulo{\sigma_1^+}
      +
      \modulo{\sigma_2^+ - (\sigma'+\sigma'')}
      & \leq &
      \O \, \kappa^2 \, \modulo{\sigma'\sigma''}
      \\
      \modulo{\sigma_1^+ - (\sigma'+\sigma'')}
      +
      \modulo{\sigma_2^+}
      & \leq &
      \O \, \kappa^2 \, \modulo{\sigma'\sigma''}
    \end{array}
  \end{equation}
\end{lemma}

\begin{proof}
  Consider first the case of interacting waves of different families,
  see Figure~\ref{fig:interactions}, left. Then, with straightforward
  computations, Lemma~\ref{lem:para} leads to
  \begin{equation}
    \label{eq:G}
    G (\sigma_1^+,\sigma_2^+,\sigma_1^-,\sigma_2^-) = 0
  \end{equation}
  where
  \begin{eqnarray*}
    G_1 (\sigma_1^+,\sigma_2^+,\sigma_1^-,\sigma_2^-)
    & = &
    \!\!\!
    \sigma_1^+ + \sigma_2^+
    - \sigma_1^- - \sigma_2^-
    \\
    G_2 (\sigma_1^+,\sigma_2^+,\sigma_1^-,\sigma_2^-)
    & = &
    \!\!\!
    \sigma_1^+ \,
    F \left(\Pi_\kappa(p_o+\sigma_1^+),\Pi_\kappa(p_o)\right)
    -
    \sigma_2^+ \,
    F \left(\Pi_\kappa(p_o+\sigma_1^+),\Pi_\kappa(p_o+\sigma_1^+
      +
      \sigma_2^+)\right)
    \\
    & &
    \!\!\!\!\!\!
    -
    \sigma_1^- \,
    F\left(\Pi_\kappa(p_o+\sigma_1^-+\sigma_2^-),
      \Pi_\kappa(p_o+\sigma_2^-)\right)
    +
    \sigma_2^- \,
    F\left(\Pi_\kappa(p_o), \Pi_\kappa(p_o + \sigma_2^-)\right)
  \end{eqnarray*}
  Note that by~\ref{it:C1}.~and~\ref{it:Ctanto}.~in
  Lemma~\ref{lem:para}, the function $G$ is of class $\C{2}$ and since
  $\Pi'_{\kappa}\left(p\right)=\kappa^{2}$ one can compute
  \begin{equation}
    \label{eq:D2G}
    \norma{D^{2}G (\sigma_1^+,\sigma_2^+,\sigma_1^-,\sigma_2^-)}_{\L\infty}
    =
    \O\kappa^{2} \,.
  \end{equation}
  Moreover, $G (0,0,0,0)= (0,0)$ and by direct computations, the
  Jacobian Matrix of $G$ with respect to $\sigma_1^+$ and $\sigma_2^+$
  computed at $(0,0,0,0)$ is
  \begin{displaymath}
    \partial_{(\sigma_1^+, \sigma_2^+)} G (0,0,0,0)
    =
    \left[
      \begin{array}{c@{\qquad}c}
        1 & 1
        \\[10pt]
        \sqrt{-\mathcal{T}' \left(\Pi_{\kappa}(p_{o})\right)}
        &
        - \sqrt{-\mathcal{T}' \left(\Pi_{\kappa}(p_{o})\right)}
      \end{array}
    \right] \,.
  \end{displaymath}
  \begin{displaymath}
    \left|\det \partial_{(\sigma_1^+, \sigma_2^+)} G (0,0,0,0)\right|
    =2\sqrt{-\mathcal{T}'\left(\bar p + \kappa^{2}\left(p_{o}-\bar
          p\right)\right)}\ge2\sqrt{-\mathcal{T}'\left(\max\left\{\bar
          p,p_{o}\right\}\right)}>0
  \end{displaymath}
  Hence, the Implicit Function Theorem ensures that~\eqref{eq:G}
  uniquely defines a map $\Sigma^\kappa$ of class $\C{2}$ such
  that~\eqref{eq:G} is equivalent to
  \begin{displaymath}
    (\sigma_1^+, \sigma_2^+)
    =
    \Sigma^\kappa (\sigma_1^-, \sigma_2^-)
  \end{displaymath}
  for all $(\sigma_1^-, \sigma_2^-)$ in a neighborhood of $(0,0)$
  which can be chosen independently of $\kappa$. Moreover,
  by~\eqref{eq:D2G},
  \begin{equation}
    \label{eq:D2Sigma}
    \norma{D^{2}\Sigma^\kappa (\sigma_1^-, \sigma_2^-)}_{\L\infty}
    \leq
    \O \kappa^{2} \,.
  \end{equation}
  By construction, the following equalities are immediate:
  \begin{displaymath}
    \Sigma^\kappa (\sigma_1, 0) = (\sigma_1,0) \,,\qquad
    \Sigma^\kappa (0, \sigma_2) = (0, \sigma_2) \,.
  \end{displaymath}
  Using~\eqref{eq:D2Sigma}, compute now
  \begin{eqnarray*}
    \modulo{\Sigma^\kappa_1 (\sigma_1^-,\sigma_2^-) - \sigma_1^-}
    & = &
    \modulo{\int_0^1
      \partial_{\sigma_2^-} \Sigma^\kappa_1 (\sigma_1^-, \theta\sigma_2^-)
      \d\theta}
    \modulo{\sigma_2^-}
    \\
    & = &
    \modulo{\int_0^1
      \left(
        \partial_{\sigma_2^-} \Sigma^\kappa_1 (\sigma_1^-, \theta\sigma_2^-)
        -
        \partial_{\sigma_2^-} \Sigma^\kappa_1 (0, \theta\sigma_2^-)
      \right)
      \d\theta} \,
    \modulo{\sigma_2^-}
    \\
    & = &
    \modulo{
      \int_0^1 \int_0^1
      \partial^2_{\sigma_1^-\sigma_2^-}
      \Sigma^\kappa_1 (\theta'\sigma_1^-, \theta\sigma_2^-)
      \d{\theta'} \d\theta}
    \modulo{\sigma_1^-\, \sigma_2^-}
    \\
    & = &
    \O \kappa^2 \modulo{\sigma_1^-\, \sigma_2^-} \,.
  \end{eqnarray*}

  \medskip

  We now consider the second estimate in~\eqref{eq:ieNonStandard},
  corresponding to the case of interacting waves both belonging to the
  second family. With the notation in Figure~\ref{fig:interactions},
  middle, we have
  \begin{equation}
    \label{eq:Gsame}
    G (\sigma_1^+, \sigma_2^+, \sigma', \sigma'') = 0
  \end{equation}
  where now
  \begin{eqnarray*}
    G_1(\sigma_1^+, \sigma_2^+, \sigma', \sigma'') & = &
    \sigma_1^+ + \sigma_2^+ - \sigma' -\sigma'' \,,
    \\
    G_2(\sigma_1^+, \sigma_2^+, \sigma', \sigma'') & = &
    \sigma_1^+ \,
    F \left(\Pi_\kappa(p_o+\sigma_1^+),\Pi_\kappa(p_o)\right)
    -
    \sigma_2^+ \,
    F \left(\Pi_\kappa(p_o+\sigma_1^+),\Pi_\kappa(p_o+\sigma_1^++\sigma_2^+)\right)
    \\
    & &
    \!\!\!
    +
    \sigma'
    F\left(\Pi_\kappa (p_o), \Pi_\kappa (p_o+\sigma')\right)
    +
    \sigma''
    F\left(\Pi_\kappa (p_o + \sigma'), \Pi_\kappa (p_o+\sigma'+\sigma'')\right) \,.
  \end{eqnarray*}
  Note that by~\ref{it:C1}, \ref{it:Ctanto} in Lemma~\ref{lem:para},
  the function $G$ is of class $\C{2}$ and since
  $\Pi'_{\kappa}\left(p\right)=\kappa^{2}$ one can compute again
  \begin{equation}
    \label{eq:D2Gsame}
    \norma{D^{2}G (\sigma_1^+,\sigma_2^+,\sigma',\sigma'')}_{\L\infty}
    =
    \O\kappa^{2} \,.
  \end{equation}
  Moreover, $G (0,0,0,0)= (0,0)$ and by direct computations, the
  Jacobian Matrix of $G$ with respect to $\sigma_1^+$ and $\sigma_2^+$
  computed at $(0,0,0,0)$ is, as before,
  \begin{displaymath}
    D_{(\sigma_1^+,\sigma_2^+)} G (0,0,0,0)
    =
    \left[
      \begin{array}{cc}
        1&1
        \\[6pt]
        \sqrt{-\mathcal{T}'\left(\Pi_\kappa (p_o)\right)}
        &
        - \sqrt{-\mathcal{T}'\left(\Pi_\kappa (p_o)\right)}
      \end{array}
    \right] \,.
  \end{displaymath}
  Hence, as before, the Implicit Function Theorem ensures
  that~\eqref{eq:Gsame} uniquely defines a map $\Sigma^\kappa$ of
  class $\C{2}$ such that~\eqref{eq:Gsame} is equivalent to
  \begin{displaymath}
    (\sigma_1^+, \sigma_2^+)
    =
    \Sigma^\kappa (\sigma', \sigma'')
  \end{displaymath}
  for all $(\sigma', \sigma'')$ in a neighborhood of $(0,0)$ which can
  be chosen independently of $\kappa$. Moreover.
  \begin{equation}
    \label{eq:D2SigmaSame}
    \norma{D^{2}\Sigma^\kappa (\sigma', \sigma'')}_{\L\infty}
    \leq
    \O \kappa^{2} \,.
  \end{equation}
  By construction, the following equalities are immediate:
  \begin{displaymath}
    \Sigma^\kappa (\sigma', 0) = (0, \sigma') \,,\qquad
    \Sigma^\kappa (0,\sigma'') = (0, \sigma'') \,.
  \end{displaymath}
  so that, using~\eqref{eq:D2SigmaSame}
  \begin{eqnarray*}
    & &
    \norma{\Sigma^\kappa (\sigma',\sigma'') - (0,\sigma'+\sigma'')}
    \\
    & = &
    \norma{
      \left(\Sigma^\kappa (\sigma',\sigma'') - (0,\sigma'+\sigma'')\right)
      -
      \left(\Sigma^\kappa (\sigma',0) - (0,\sigma')\right)
    }
    \\
    & = &
    \norma{\sigma''
      \int_0^1
      \left(
        \partial_{\sigma''} \Sigma^\kappa (\sigma', \theta\sigma'') - (0,1)
      \right)
      \d\theta
    }
    \\
    & = &
    \norma{\sigma''
      \int_0^1
      \left(
        \left(
          \partial_{\sigma''} \Sigma^\kappa (\sigma', \theta\sigma'') - (0,1)
        \right)
        -
        \left(\partial_{\sigma''}\Sigma^\kappa (0,\theta\sigma'') - (0,1)\right)
      \right)\d\theta
    }
    \\
    & = &
    \modulo{\sigma' \sigma''}
    \norma{
      \int_0^1 \int_0^1
      \partial^2_{\sigma'\sigma''}\Sigma^\kappa (\theta'\sigma', \theta\sigma'')
      \d{\theta} \d{\theta'}
    }
    \\
    & = &
    \O \, \kappa^2 \, \modulo{\sigma' \, \sigma''} \,,
  \end{eqnarray*}
  completing the proof of the second estimate
  in~\eqref{eq:ieNonStandard}. The case of two interacting waves both
  belonging to the first family in Figure~\ref{fig:interactions},
  right, is entirely similar.
\end{proof}

The estimates on the waves' sizes in the case of interactions
involving the interfaces are as follows.

\begin{lemma}
  \label{lem:InterInter}
  There exist positive $\bar\delta$, $c$ and $\kappa_*<1$ such that,
  if all the interacting waves in Figure~\ref{fig:ie} have strength
  less than $\bar\delta$, then the following estimates hold:
  \begin{equation}
    \label{eq:ieFig2uno}
    \begin{array}{rcr@{\;}c@{\;}r@{}}
      \modulo{\sigma_1^+}
      & \leq &
      \O \, \kappa \, \modulo{\sigma_1^-}
      & + &
      \left(1 + \O (\kappa+\bar\delta)\right) \modulo{\sigma_2^-}
      \\[6pt]
      \modulo{\sigma_2^+}
      & \leq &
      (1 - c\, \kappa) \, \modulo{\sigma_1^-}
      & + &
      \left( 2+ \O \bar\delta\right)
      \modulo{\sigma_2^-}
    \end{array}
  \end{equation}
  uniformly for all $\kappa \in \left]0, \kappa_*\right[$. Moreover:
  \begin{equation}
    \label{eq:come}
    \sigma_1^+ + \sigma_2^+ = \sigma_1^- + \sigma_2^-
    \qquad \mbox{ and } \qquad
    \left\{
      \begin{array}{rcl}
        \sigma_1^- = 0 & \Rightarrow &
        \left\{
          \begin{array}{l}
            \modulo{\sigma_2^+} - \modulo{\sigma_1^+} = \modulo{\sigma_2^-} \,,
            \\
            \sigma_2^- \, \sigma_2^+ \geq 0
            \mbox{ and }
            \sigma_2^- \, \sigma_1^+ \leq 0 \,.
          \end{array}
        \right.
        \\[18pt]
        \sigma_2^- = 0 & \Rightarrow&
        \left\{
          \begin{array}{l}
            \modulo{\sigma_1^+} + \modulo{\sigma_2^+} = \modulo{\sigma_1^-} \,,
            \\
            \sigma_1^- \, \sigma_1^+ \geq 0
            \mbox{ and }
            \sigma_1^- \, \sigma_2^+ \geq 0 \,.
          \end{array}
        \right.
      \end{array}
    \right.
  \end{equation}
\end{lemma}

\begin{figure}[!h]
  \centering
  \begin{tikzpicture}[line width=0.5mm,yscale=2, xscale=1.5]
    \draw[color=white,pattern=horizontal lines] (0,-1) -- (1.5,-1) --
    (1.5,1) -- (0,1); \node at (-1.6,-1.1) {$\sigma_{2}^{-}$}; \node
    at (1.6,1.1) {$\sigma_{2}^{+}$}; \node at (-1.6,0)
    {$\left(p_{o},v_{o}\right)$}; \node at (-1.6,1.1)
    {$\sigma_{1}^{+}$}; \node at (1.5,-1.1) {$\sigma_{1}^{-}$};
    \draw[-] (1.5,-1) -- (0,0); \node at (0.6,1.1) {Liquid phase};
    \draw[-] (-1.5,-1) -- (0,0); \draw[-] (0,0) -- (1.5,1); \draw[-]
    (0,0) -- (-1.5,1); \draw[line width=0.2] (0,-1) -- (0,1);
  \end{tikzpicture}
  \caption{Notation for the proof of~Lemma~\ref{lem:InterInter}:
    $\sigma_1^-$, coming from the liquid phase, and $\sigma_2^-$,
    coming from the gas phase, hit against the phase boundary
    generating $\sigma_2^+$ in the liquid phase and $\sigma_1^+$ in
    the gas phase.}
  \label{fig:ie}
\end{figure}
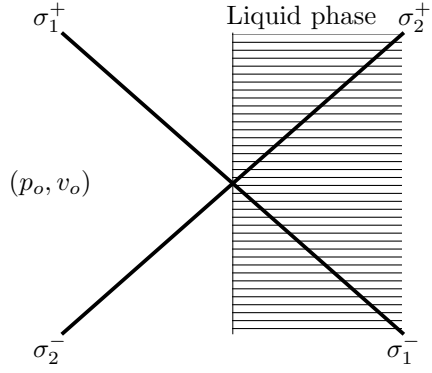

\begin{proof}
  In the present case, we have
  \begin{equation}
    \label{eq:GInter}
    G (\sigma_1^+, \sigma_2^+, \sigma_1^-, \sigma_2^-)
    =0
  \end{equation}
  where
  \begin{eqnarray*}
    G_1 (\sigma_1^+, \sigma_2^+, \sigma_1^-, \sigma_2^-)
    & = &
    \sigma_1^+ + \sigma_2^+ - \sigma_1^- - \sigma_2^-
    \\
    G_2 (\sigma_1^+, \sigma_2^+, \sigma_1^-, \sigma_2^-)
    & = &
    \sigma_2^- \, F_g (p_o, p_o+\sigma_2^-)
    -
    \kappa \, \sigma_1^- \,
    F\left(
      \Pi_\kappa (p_o+\sigma_1^-+\sigma_2^-),
      \Pi_\kappa (p_o+\sigma_2^-)
    \right)
    \\
    & &
    +
    \sigma_1^+ \, F_g (p_o+\sigma_1^+,p_o)
    -
    \kappa \, \sigma_2^+ \,
    F\left(
      \Pi_\kappa (p_o+\sigma_1^+),
      \Pi_\kappa (p_o+\sigma_1^++\sigma_2^+)
    \right)
  \end{eqnarray*}
  with $G (0,0,0,0) = 0$ and Jacobian matrix
  \begin{displaymath}
    \partial_{\left(\sigma_1^+,\sigma_2^+\right)} G (0,0,0,0)
    =
    \left[
      \begin{array}{cc}
        1 & 1
        \\[6pt]
        \sqrt{-\mathcal{T}_g' (p_o)} &
        - \kappa \sqrt{-\mathcal{T}'\left(\Pi_\kappa (p_o)\right)}
      \end{array}
    \right]
  \end{displaymath}
  \begin{displaymath}
    \left|\det \partial_{\left(\sigma_1^+,\sigma_2^+\right)} G
      (0,0,0,0)\right|
    =\sqrt{-\mathcal{T}_g' (p_o)} +\kappa
    \sqrt{-\mathcal{T}'\left(\Pi_\kappa (p_o)\right)}\ge
    \sqrt{-\mathcal{T}_g' (p_o)}>0
  \end{displaymath}
  and moreover
  \begin{displaymath}
    D^2G (\sigma_1^+,\sigma_2^+,\sigma_1^-,\sigma_2^-) = \O
  \end{displaymath}
  uniformly in $\kappa$, which allows to apply the Implicit Function
  Theorem in the same neighborhood of radius $\bar \delta$ for all
  small $\kappa$, yielding a map $ \Sigma^\kappa (\sigma_1^-,
  \sigma_2^-)= \left(\sigma_1^+ , \sigma_2^+\right)$ such that
  \begin{equation}
    \label{eq:secdersi}
    D^2\Sigma^\kappa (\sigma_1^-,\sigma_2^-) = \O
  \end{equation}
  locally in $(\sigma_1^-, \sigma_2^-)$ and uniformly in
  $\kappa$. Moreover,
  \begin{displaymath}
    \begin{array}{@{}cl@{}}
      &
      D\Sigma^\kappa (0,0)
      \\
      = &
      -
      \left[D_{(\sigma_1^+, \sigma_2^+)} G (0,0,0,0)\right]^{-1}
      \;
      D_{(\sigma_1^-,\sigma_2^-)} G (0,0,0,0)
      \\
      = &
      \frac{1}{\sqrt{-\mathcal{T}'_{g} (p_o)}
        +
        \kappa \sqrt{-\mathcal{T}'\left(\Pi_\kappa (p_o)\right)}}
      \left[
        \begin{array}{@{}cc@{}}
          \kappa \sqrt{-\mathcal{T}'\left(\Pi_\kappa (p_o)\right)}
          &
          1
          \\
          \sqrt{-\mathcal{T}'_{g} (p_o)}
          & -1
        \end{array}
      \right]
      \left[
        \begin{array}{@{}cc@{}}
          1
          &
          1
          \\
          \kappa \sqrt{-\mathcal{T}'\left(\Pi_\kappa (p_o)\right)}
          &
          -\sqrt{-\mathcal{T}' _{g}(p_o)}
        \end{array}
      \right]
      \\
      = &
      \frac{1}{\sqrt{-\mathcal{T}'_{g} (p_o)}
        +
        \kappa \sqrt{-\mathcal{T}'\left(\Pi_\kappa (p_o)\right)}}
      \left[
        \begin{array}{@{}c@{\,}c@{}}
          2 \, \kappa \sqrt{-\mathcal{T}'\left(\Pi_\kappa (p_o)\right)}
          &
          -\sqrt{-\mathcal{T}' _{g}(p_o)}
          +
          \kappa \sqrt{-\mathcal{T}'\left(\Pi_\kappa (p_o)\right)}
          \\
          \sqrt{-\mathcal{T}'_{g} (p_o)}
          -
          \kappa \sqrt{-\mathcal{T}'\left(\Pi_\kappa (p_o)\right)}
          &
          2 \, \sqrt{-\mathcal{T}'_{g} (p_o)}
        \end{array}
      \right]
    \end{array}
  \end{displaymath}
  which shows that the following bound $D\Sigma^\kappa (0,0) = \O$
  hold uniformly in $\kappa$. This, together with~\eqref{eq:secdersi},
  implies
  \begin{displaymath}
    D \Sigma^\kappa (\sigma_1^-, \sigma_2^-) = \O \,,
  \end{displaymath}
  so that
  \begin{equation}
    \label{eq:l1}
    \Sigma^\kappa (\sigma_1^-, \sigma_2^-)
    = \O \left(\modulo{\sigma_1^-} + \modulo{\sigma_2^-}\right)
  \end{equation}
  since $\Sigma^\kappa (0,0)=0$.  Solve now $G_2 (\sigma_1^+,
  \sigma_2^+, \sigma_1^-, \sigma_2^-)=0$ for $\sigma_1^+$, use the
  bound~\eqref{eq:l1} and the estimates $\modulo{\sigma_1^-},
  \modulo{\sigma_2^-} < \bar \delta$ to obtain:
  \begin{eqnarray}
    \label{eq:daqui}
    \sigma_1^+
    & = &
    -
    \frac{F_g (p_o,p_o+\sigma_2^-)}{F_g (p_o+\sigma_1^+, p_o)} \, \sigma_2^-
    +
    \frac{F\left(\Pi_\kappa (p_o+\sigma_1^-+\sigma_2^-), \Pi_\kappa (p_o+\sigma_2^-)\right)}{F_g (p_o+\sigma_1^+, p_o)} \,
    \kappa \, \sigma_1^-
    \\
    \label{eq:aqui}
    & &
    +
    \frac{F\left(\Pi_\kappa (p_o+\sigma_1^+), \Pi_\kappa (p_o+\sigma_1^++\sigma_2^+)\right)}{F_g (p_o+\sigma_1^+, p_o)} \,
    \kappa \, \sigma_2^+
    \\
    \nonumber
    & \leq &
    \left(1+\O \bar\delta\right) \modulo{\sigma_2^-}
    +
    \O \, \kappa \, \modulo{\sigma_1^-}
    +
    \O \, \kappa \left(\modulo{\sigma_1^-} + \modulo{\sigma_2^-}\right)
    \\
    \nonumber
    & = &
    \O \, \kappa \, \modulo{\sigma_1^-}
    +
    \left(1 + \O (\kappa+\bar \delta)\right) \modulo{\sigma_2^-}
  \end{eqnarray}
  which gives the first estimate in~\eqref{eq:ieFig2uno}. To obtain
  the second one, use $G_1 (\sigma_1^+, \sigma_2^+, \sigma_1^-,
  \sigma_2^-) = 0$ and~\eqref{eq:daqui}--\eqref{eq:aqui}:
  \begin{eqnarray}
    \label{eq:s1}
    \!\!\!\!\!\!\!\!\!
    & &
    \left(
      1
      +
      \kappa\,
      \frac{F\left(\Pi_\kappa (p_o+\sigma_1^+), \Pi_\kappa (p_o+\sigma_1^++\sigma_2^+)\right)}{F_g (p_o+\sigma_1^+, p_o)}
    \right)
    \sigma_2^+
    \\
    \label{eq:s2}
    \!\!\!\!\!\!\!\!\!
    & = &
    \left(
      1
      -
      \kappa \,
      \frac{F\left(\Pi_\kappa (p_o+\sigma_1^-+\sigma_2^-), \Pi_\kappa (p_o+\sigma_1^-)\right)}{F_g (p_o+\sigma_1^+, p_o)}
    \right)
    \sigma_1^-
    +
    \left(
      1
      +
      \frac{F_g (p_o,p_o+\sigma_2^-)}{F_g (p_o+\sigma_1^+, p_o)}
    \right)
    \sigma_2^-
  \end{eqnarray}
  which implies the second in~\eqref{eq:ieFig2uno}, since for a
  suitable $c>0$,
  \begin{eqnarray*}
    \frac{F\left(\Pi_\kappa (p_o+\sigma_1^+), \Pi_\kappa (p_o+\sigma_1^++\sigma_2^+)\right)}{F_g (p_o+\sigma_1^+, p_o)}
    & \geq &
    c \,,
    \\
    \frac{F\left(\Pi_\kappa (p_o+\sigma_1^-+\sigma_2^-), \Pi_\kappa (p_o+\sigma_1^-)\right)}{F_g (p_o+\sigma_1^+, p_o)}
    & \geq &
    c \,,
    \\
    \frac{F_g (p_o,p_o+\sigma_2^-)}{F_g (p_o+\sigma_1^+, p_o)}
    & \leq &1 + \O \, \bar\delta \,.
  \end{eqnarray*}

  To prove~\eqref{eq:come}, note that in the case $\sigma_1^- = 0$,
  \eqref{eq:s1}--\eqref{eq:s2} imply that $\sigma_2^+$ and
  $\sigma_2^-$ have the same sign. On the other hand,
  by~\eqref{eq:daqui}--\eqref{eq:aqui}
  \begin{displaymath}
    \sigma_1^+
    =
    \left(
      -
      \frac{F (p_o,p_o+\sigma_2^-)}{F (p_o+\sigma_1^+, p_o)}
      +
      \O \, \kappa
    \right)
    \sigma_2^-
  \end{displaymath}
  so that $\sigma_1^+$ and $\sigma_2^-$ have different signs whenever
  $\kappa$ is sufficiently small, proving the first equality on the
  right in~\eqref{eq:come}.

  Assume now that $\sigma_2^- = 0$, so that $\sigma_1^+ + \sigma_2^+ =
  \sigma_1^-$. By~\eqref{eq:s1}--\eqref{eq:s2}, $\sigma_1^-$ and
  $\sigma_2^+$ have the same sign for $\kappa$ small. The second
  inequality in~\eqref{eq:ieFig2uno} then ensures that
  $\modulo{\sigma_2^+} < \modulo{\sigma_1^-}$ and hence also
  $\sigma_1^+$ has the same sign of $\sigma_1^-$ and $\sigma_2^+$.
\end{proof}

Remark that a wave refracted at the phase boundary remains of the same
type, whereas the reflected wave changes type when it comes from the
liquid and remains of the same type when it comes from the gas, see
Table~\ref{tab:signs} and~\eqref{eq:come}.

\medskip

At any fixed positive time $t$, the approximate solution is a
piecewise constant function $u^\epsilon (t) = \sum_\alpha u_\alpha \,
\chi_{\strut [z_\alpha, z_{\alpha+1}[}$. If $t$ is not an interaction
time, we denote by $\sigma_\alpha$ the size of the wave supported at
$z_\alpha$ and introduce the potentials
\begin{equation}
  \label{eq:Upsilon}
  \begin{array}{c}
    \displaystyle
    V_{\mathcal{G}_{in}} = \sum_{\alpha \in \mathcal{G}_{in}} \modulo{\sigma_\alpha}
    \qquad\qquad
    V_{\mathcal{G}_{out}}  =  \sum_{\alpha \in \mathcal{G}_{out}} \modulo{\sigma_\alpha}
    \qquad\qquad
    V_{\mathcal{L}} = \sum_{\alpha \in \mathscr{L}} \modulo{\sigma_\alpha}
    \\[18pt]
    \displaystyle
    Q_{\mathcal{G}}  =
    \sum_{(\alpha, \beta) \in \mathcal{A}_{\mathcal{G}}}
    \modulo{\sigma_\alpha \, \sigma_\beta}
    \qquad\qquad
    Q_{\mathcal{L}}  =
    \sum_{(\alpha, \beta) \in \mathcal{A}_{\mathcal{L}}}
    \modulo{\sigma_\alpha \, \sigma_\beta}
    \\[18pt]
    \displaystyle
    \Upsilon
    =
    K_{in} \, V_{\mathcal{G}_{in}}
    +
    V_{\mathcal{G}_{out}}
    +
    K_{\mathcal{L}} \, V_{\mathcal{L}}
    +
    H_{\mathcal{G}} \, Q_{\mathcal{G}}
    +
    \kappa^2 \, H_{\mathcal{L}} \, Q_{\mathcal{L}} \,,
  \end{array}
\end{equation}
where $K_{in}, K_{\mathcal{L}}, H_{\mathcal{G}}$ and $H_{\mathcal{L}}$
are constants independent of $\kappa$ to be precisely defined below. Above, we denoted\\
\begin{tabular}{cp{0.9\textwidth}}
  $\mathcal{G}_{in}$
  &
  $2$-waves supported in $\left]-\infty, 0\right[$ and $1$-waves supported in $\left]m, +\infty\right[$.
  \\
  $\mathcal{G}_{out}$
  &
  $1$-waves supported in $\left]-\infty, 0\right[$ and $2$-waves supported in $\left]m, +\infty\right[$.
  \\
  $\mathscr{L}$
  &
  all waves supported in the liquid phase $\mathcal{L}$.
  \\
  $\mathcal{A}_{\mathcal{G}}$
  &
  pairs of approaching waves supported in the gas phase.
  \\
  $\mathcal{A}_{\mathcal{L}}$
  &
  pairs of approaching waves supported in the liquid phase.
\end{tabular}
\\
Here, we define as \emph{approaching} two waves both supported in the
same interval $\left]-\infty, 0\right[$, $]0,m[$ or $\left]m,
  +\infty\right[$, either of the same family and when one of the two
is a shock, or of different families with the one of the first family
on the right.

\begin{lemma}
  \label{lem:Upsilon}
  There exist weights $K_{in}, K_{\mathcal{L}}, H_{\mathcal{G}}$ and
  $H_{\mathcal{L}}$, all greater than $1$, $\kappa_* \in \left]0, 1
  \right[$ and a positive $\bar\delta$ such that, for all $\kappa \in
  \left]0, \kappa_*\right[$ and piecewise constant initial data
  $\tilde u^\epsilon$ with the corresponding approximate solution
  $u^\epsilon$ constructed by the algorithm above satisfying $\Upsilon
  (u^\epsilon(0+)) < \bar\delta$, the function $t \to
  \Upsilon\left(u^\epsilon (t)\right)$ is non increasing. Moreover,
  calling $\sigma_\alpha, \sigma_\beta$ the waves interacting at time
  $\bar t$ and point $\bar z$, with $\sigma_\alpha$ coming from the
  left, the following estimates hold:
  \begin{equation}
    \label{eq:1}
    \begin{array}{r@{\;}c@{\;}l@{\qquad\quad}rcl}
      \bar z & \in & \mathcal{G}
      &
      \Delta \Upsilon  & \leq &
      - \modulo{\sigma_\alpha \, \sigma_\beta}
      \\
      \bar z & = & 0
      &
      \Delta \Upsilon  & \leq &
      -\modulo{\sigma_\alpha} - \kappa \, \modulo{\sigma_\beta}
      \\
      \bar z & \in & \mathcal{L}
      &
      \Delta \Upsilon  & \leq &
      - \kappa^2 \, \modulo{\sigma_{\alpha} \, \sigma_\beta}
      \\
      \bar z & = & m
      &
      \Delta \Upsilon  & \leq &
      -\kappa \, \modulo{\sigma_\alpha} - \modulo{\sigma_\beta} \,.
    \end{array}
  \end{equation}
\end{lemma}

\begin{proof}
  Denote by $C$, with $C > 1$, a positive constant bounding from above
  all $\O$ appearing in~\eqref{eq:ieStandard},
  \eqref{eq:ieNonStandard} and~\eqref{eq:ieFig2uno}. Choose
  $\bar\delta>0$ such that $\bar\delta < 1 / (2C)$, and $\tilde
  u^{\varepsilon}$ such that
  $\Upsilon\left(u^{\varepsilon}(0+)\right)<\bar \delta$.

  Suppose that at time $\bar t$ there is an interaction and that
  $\Upsilon\left(u^{\varepsilon}(\bar t-)\right)<\bar\delta$.
  Consider the different interactions separately. Begin with an
  interaction in $\mathcal{G}$, as in Figure~\ref{fig:interactions},
  using~\eqref{eq:ieStandard} and definitions~\eqref{eq:Upsilon}:
  \begin{displaymath}
    \begin{array}{rcl@{\qquad\qquad}rcl}
      \Delta V_{\mathcal{G}_{in}}
      & \leq &
      C \, \modulo{\sigma_\alpha \, \sigma_\beta}
      &
      \Delta Q_{\mathcal{G}}
      & \leq &
      C \, \modulo{\sigma_\alpha \, \sigma_\beta} \bar\delta
      -
      \modulo{\sigma_\alpha \, \sigma_\beta}
      \leq
      -\frac{1}{2} \, \modulo{\sigma_\alpha \, \sigma_\beta}
      \\
      \Delta V_{\mathcal{G}_{out}}
      & \leq &
      C \, \modulo{\sigma_\alpha \, \sigma_\beta}
      &
      \Delta Q_{\mathcal{L}}
      & = &
      0
      \\
      \Delta V_{\mathcal{L}}
      & = &
      0
      &
      \Delta \Upsilon
      & \leq &
      (C \, K_{in} + C - \frac{1}{2}\, H_{\mathcal{G}})
      \modulo{\sigma_\alpha \, \sigma_\beta} \,.
    \end{array}
  \end{displaymath}
  Consider an interaction in the liquid phase, as in
  Figure~\ref{fig:interactions}, using~\eqref{eq:ieNonStandard} and
  definitions~\eqref{eq:Upsilon}:
  \begin{displaymath}
    \begin{array}{rcl@{\qquad\qquad}rcl}
      \Delta V_{\mathcal{G}_{in}}
      & = &
      0
      &
      \Delta Q_{\mathcal{G}}
      & = &
      0
      \\
      \Delta V_{\mathcal{G}_{out}}
      & = &
      0
      &
      \Delta Q_{\mathcal{L}}
      & \leq &
      (C \, \kappa^2 \, \bar\delta - 1)
      \modulo{\sigma_\alpha \, \sigma_\beta}
      \leq
      -\frac{1}{2} \, \modulo{\sigma_\alpha \, \sigma_\beta}
      \\
      \Delta V_{\mathcal{L}}
      & \leq &
      C \, \kappa^2 \, \modulo{\sigma_\alpha \, \sigma_\beta}
      &
      \Delta \Upsilon
      & \leq &
      \kappa^2 \, (C \, K_{\mathcal{L}} - \frac{1}{2} \, H_{\mathcal{L}})
      \modulo{\sigma_\alpha \, \sigma_\beta}
    \end{array}
  \end{displaymath}
  Consider now the case $\bar z = 0$, the case $\bar z = m$ being
  entirely analogous. By~\eqref{eq:ieFig2uno}, for $\kappa+\bar\delta$
  sufficiently small so that $C (\kappa+\bar\delta) < 1$, it follows,
  using definitions~\eqref{eq:Upsilon}, that:
  \begin{displaymath}
    \begin{array}{rcl@{\qquad\qquad}rcl}
      \Delta V_{\mathcal{G}_{in}}
      & \leq &
      -\modulo{\sigma_\alpha}
      &
      \Delta Q_{\mathcal{G}}
      & \leq &
      2 \, \bar\delta \, \modulo{\sigma_\alpha}
      + C \, \kappa \, \bar\delta \, \modulo{\sigma_\beta}
      \\
      \Delta V_{\mathcal{G}_{out}}
      & \leq &
      2 \modulo{\sigma_\alpha} + C \, \kappa \modulo{\sigma_\beta}
      &
      \Delta Q_{\mathcal{L}}
      & \leq &
      3\, \bar\delta \, \modulo{\sigma_\alpha} + \bar\delta \, \modulo{\sigma_\beta}
      \\
      \Delta V_{\mathcal{L}}
      & \leq &
      3 \modulo{\sigma_\alpha} - c\, \kappa \, \modulo{\sigma_\beta}
      &
      \Delta \Upsilon
      & \leq &
      \left[
        2
        -
        K_{in}
        +
        3 K_{\mathcal{L}}
        +
        (2H_{\mathcal{G}} + 3 \kappa^2 H_{\mathcal{L}}) \bar\delta
      \right]
      \modulo{\sigma_\alpha}
      \\
      & & & & &
      +
      \kappa
      \left[
        C
        -
        c\, K_{\mathcal{L}}
        +
        (C \, H_{\mathcal{G}} + \kappa \, H_{\mathcal{L}}) \bar\delta
      \right]
      \modulo{\sigma_\beta}
    \end{array}
  \end{displaymath}
  To complete the proof, observe that choosing
  \begin{enumerate}
  \item $K_{\mathcal{L}}$ so that $C - c \, K_\mathcal{L} \leq -2$;
  \item $K_{in}$ so that $2-K_{in}+ 3 K_{\mathcal{L}} \leq -2$;
  \item $H_{\mathcal{G}}$ so that $C(1+K_{in}) - \frac12\,
    H_{\mathcal{G}} \leq -1$;
  \item $H_\mathcal{L}$ so that $C \, K_{\mathcal{L}} - \frac12\,
    H_{\mathcal{L}} \leq -1$;
  \item $\bar\delta$ so that $(C\, H_{\mathcal{G}} + H_{\mathcal{L}})
    \bar\delta \leq 1$ and $(2 H_{\mathcal{G}} + 3 H_{\mathcal{L}})
    \bar\delta \leq 1$.
  \end{enumerate}
  \noindent ensures that~\eqref{eq:1} holds. The proof is concluded by
  induction on the interaction times.
\end{proof}

\begin{lemma}
  \label{lem:nocluster}
  With the algorithm defined above, if the piecewise constant initial
  datum $\tilde u^\epsilon$ is chosen so that $\Upsilon
  (u^\epsilon(0+)) < \bar\delta$, with $\bar\delta$ as in
  Lemma~\ref{lem:Upsilon}, ($u^\epsilon(t)$ being the approximate
  solution constructed above) then there exists no cluster point of
  interaction points.
\end{lemma}

\begin{proof}
  By contradiction, call $t_*$ the first time at which a cluster point
  $(t_*, z_*)$ of interaction points appears.

  First, assume that $z_* \neq 0$ and $z_* \neq m$. Call $\mathcal{U}$
  a neighborhood of $(t_*, z_*)$ not intersecting the interfaces $z
  \in \{0,m\}$. The interactions where there are more than one
  outgoing waves of the same family are those where
  \begin{itemize}
  \item two waves of the same family hit against each other
    originating a rarefaction fan of the other family of total size
    bigger than $\epsilon$; and
  \item a wave hits an interface, resulting in a new reflected
    rarefaction larger than $\epsilon$ which is eventually split as it
    reaches the boundary of the strip.
  \end{itemize}
  Because of the estimates~\eqref{eq:ieStandard},
  \eqref{eq:ieNonStandard}, \eqref{eq:ieFig2uno} and~\eqref{eq:1}, at
  any of these interactions $\Delta \Upsilon \leq -
  \frac{\kappa}{C}\varepsilon$. Hence, these interactions may take
  place only a finite number of times. An application
  of~\cite[Lemma~2.5]{relaxation} contradicts the existence of $(t_*,
  z_*)$.

  Assume now $z_* = 0$, the case $z_* = m$ being entirely
  equivalent. For a small positive $\eta$, choose a trapezoid
  $\mathcal{N}_\eta$ contained in $\mathcal{I}_\epsilon^-$ of the form
  \begin{displaymath}
    \mathcal{N}_\eta
    =
    \left\{
      (t,z) \in \mathcal{I}^-
      \colon
      t \in \left]t_* - \eta, t_* \right[
      \mbox{ and }
      \modulo{\frac{z-z_*}{t-t_*-\eta}} \leq 2
    \right\}.
  \end{displaymath}
  By construction, finitely many waves cross the lower side of
  $\mathcal{N}_\eta$ and no wave may enter $\mathcal{N}_\eta$ along
  the two sides. Inside $\mathcal{N}_\eta$, any wave can generate
  another wave at most once, when it hits the interface $z =
  z_*$. Inside $\mathcal{N}_\eta$ waves propagate with speed either
  $1$ or $-1$ and at interactions between waves with different speeds,
  no new wave is produced. Hence, the total number of interaction
  points inside $\mathcal{N}_\eta$ is finite. This contradicts the
  existence of a cluster point of interaction points.
\end{proof}

To ensure that the value of the functional at $t = 0+$ is sufficiently
small in order that all the above interaction estimates hold true, we
need some conditions on the total variation of the initial data. The
standard estimates on the solution of the Riemann problem
(see~\cite[Chapter~5]{BressanLectureNotes}) imply that, in the gas, it
is sufficient that the initial datum has sufficiently small total
variation. On the other hand, in the liquid, the estimates on the
Riemann problem depend on the small parameter $\kappa$, as shown
in~\eqref{eq:es1}, see also Figure~\ref{fig:estriemann}.  All this
justifies the introduction of the weighted total
variation~\eqref{eq:wtv}.

\begin{lemma}
  \label{lem:wtw}
  Consider $\bar \delta$ as defined in Lemma~\ref{lem:RS_liquid} and
  let $\left(\tilde p^\epsilon, \tilde v^\epsilon\right)=\tilde
  u^\epsilon \colon \reali \to \reali^+ \times \reali$ be piecewise
  constant, continuous at $z=0$ and $z=m$ such that $\norma{\tilde
    p^\epsilon - p_o}_{\L\infty} < \bar\delta$. If $u^{\varepsilon}$
  is the approximate solution constructed above, then, there exists a
  positive $C$, which can be chosen bounding from above all $\O$
  appearing in~\eqref{eq:ieStandard}, \eqref{eq:ieNonStandard}
  and~\eqref{eq:ieFig2uno}, such that
  \begin{displaymath}
    \frac{1}{C} \; {\wtv}_\kappa (\tilde u^\epsilon)
    \leq
    \Upsilon (u^\epsilon(0+))
    \leq
    C \; {\wtv}_\kappa (\tilde u^\epsilon)
  \end{displaymath}
  with ${\wtv}_\kappa$ defined as in~\eqref{eq:wtv} and $\Upsilon$ as
  in~\eqref{eq:Upsilon} with the weights $K_{in}, K_{\mathcal{L}},
  H_{\mathcal{G}}$ and $H_{\mathcal{L}}$ chosen as in
  Lemma~\ref{lem:Upsilon}.
\end{lemma}

\begin{proof}
  Let $\sigma_{\alpha}$ be the sizes of the waves in
  $u^{\varepsilon}(0+)$ and $z_{\alpha}$ be their locations.  Consider
  the estimate on the left.  The strength of a wave is the absolute
  value of the pressure difference between the states on its sides,
  therefore, because of the weights' choice in Lemma~\ref{lem:Upsilon}
  (they are all greater than $1$), we have
  \begin{displaymath}
    \tv (\tilde p^\epsilon)
    \leq
    \tv (p^\epsilon(0+))
    =
    \sum_\alpha \modulo{\sigma_\alpha}
    \leq
    \Upsilon (u^\epsilon(0+)) \,.
  \end{displaymath}
  The slopes of Lax curves in the gas do not depend on
  $\kappa$~\eqref{eq:LaxGas}, hence, along a Lax curve, the jump in
  the speed is uniformly controlled by the jump in the pressure:
  \begin{displaymath}
    \tv (\tilde v^\epsilon; \mathcal{G})
    \le
    \tv (v^\epsilon(0+); \mathcal{G})
    \le
    \O \sum_\alpha \modulo{\sigma_\alpha}
    \le
    \O \Upsilon (u^\epsilon(0+)) \,.
  \end{displaymath}
  Finally, in the liquid we use~\eqref{eq:para} which shows that along
  a Lax curve in the liquid, the jump in the speed is controlled by
  $\kappa$ times the jump in the pressure:
  \begin{displaymath}
    \tv (\tilde v^\epsilon; \mathcal{L})
    \le
    \tv (v^\epsilon(0+); \mathcal{L})
    \le
    \O \kappa\sum_\alpha \modulo{\sigma_\alpha}
    \le
    \O \kappa\Upsilon (u^\epsilon(0+)) \,.
  \end{displaymath}
  This concludes the proof of the left estimate.

  Passing to the right inequality, recall the usual bound $\Upsilon
  (u^\epsilon(0+)) \leq \O \sum_\alpha \modulo{\sigma_\alpha}$ which
  clearly holds also for $\Upsilon$ as defined in~\eqref{eq:Upsilon}.
  Proceed using the classical estimate for the solutions to the
  Riemann problems in the gas and~\eqref{eq:es1} in the liquid:
  \begin{eqnarray*}
    \Upsilon (u^\epsilon(0+))
    & = &
    \O \sum_\alpha \modulo{\sigma_\alpha}
    \\
    & = &
    \O \left(
      \sum_{z_\alpha \in \pint{\mathcal{G}}}\modulo{\sigma_\alpha}
      +
      \sum_{z_\alpha \in \mathcal{L}}\modulo{\sigma_\alpha}
    \right)
    \\
    & = &
    \O \Bigg(
    \sum_{z_\alpha \in \pint{\mathcal{G}}}
    \left(
      \modulo{\tilde p^\epsilon (z_\alpha+) - \tilde p^\epsilon (z_\alpha-)}
      +
      \modulo{\tilde v^\epsilon (z_\alpha+) - \tilde v^\epsilon (z_\alpha-)}
    \right)
    \\
    & &
    \qquad\qquad
    +
    \sum_{z_\alpha \in \mathcal{L}}
    \left(
      \modulo{\tilde p^\epsilon (z_\alpha+) - \tilde p^\epsilon (z_\alpha-)}
      +
      \frac{1}{\kappa}
      \modulo{\tilde v^\epsilon (z_\alpha+) - \tilde v^\epsilon (z_\alpha-)}
    \right)
    \Bigg)
    \\
    & = &
    \O \, {\wtv}_\kappa (\tilde u^\epsilon) \,,
  \end{eqnarray*}
  completing the proof.
\end{proof}

\begin{proposition}
  \label{prop:approxsol}
  Fix a positive pressure $p_o$ and let $P^g, P$ satisfy~\textbf{(P)}.
  There exist constants $\delta, \Delta, L, \kappa_* > 0$, with
  $\kappa_* < 1$, such that, for any $\kappa \in \left]0,
    \kappa_*\right[$, for any piecewise constant initial datum $\tilde
  u = (\tilde p, \tilde v)$, continuous at the points $z=0$, $z=m$,
  satisfying $\wtv_{\kappa} (\tilde u) \le \delta$ and $\norma{\tilde
    p - p_o}_{\L\infty} \le \delta$, the wave front tracking
  approximate solution $u^{\kappa,\varepsilon} = (p^{\kappa,\epsilon},
  v^{\kappa,\epsilon})$ to the Cauchy problem
  for~\eqref{eq:finalequation} can be constructed for all times $t\geq
  0$. Moreover, given the specific volume as $\tau^{\kappa,\epsilon}
  \left(t,z\right) = \mathcal{T}_{\kappa} \left(z,p^{\kappa,\epsilon}
    \left(t,z\right) \right)$, the following estimates hold.

  \medskip

  \noindent For any $t,t_{1},t_{2}\geq 0$
  \begin{equation}
    \label{eq:timeestimates}
    \begin{array}{@{}l@{\,}c@{\,}rl@{\,}c@{\,}r@{}}
      {\wtv}_{\kappa} \left(u^{\kappa,\varepsilon}(t,\cdot)\right)
      & \leq &
      \Delta,
      &&
      \\[10pt]
      \tv\left(p^{\kappa,\varepsilon}(t,\cdot), \mathcal{L}\right)
      & \leq &
      \Delta,
      &
      \int_{\mathcal{L}} \left|p^{\kappa,\varepsilon}
        (t_{2},z)-p^{\kappa,\varepsilon}(t_{1},z)\right|\d{z}
      & \leq &
      \frac{1}{\kappa}L\left|t_{2}-t_{1}\right|,
      \\[10pt]
      \tv\left(v^{\kappa,\varepsilon}(t,\cdot),\mathcal{L}\right)
      & \leq &
      \kappa\Delta,
      &\int_{\mathcal{L}}\left|v^{\kappa,\varepsilon}
        (t_{2},z)-v^{\kappa,\varepsilon}(t_{1},z)\right|\d{z}
      & \leq &   L\left|t_{2}-t_{1}\right|,
      \\[10pt]
      \tv\left(\tau^{\kappa,\varepsilon}(t,\cdot), \mathcal{L}\right)
      & \leq &
      \kappa^{2}\Delta,
      &\int_{\mathcal{L}}\left|\tau^{\kappa,\varepsilon}
        (t_{2},z)-\tau^{\kappa,\varepsilon}(t_{1},z)\right|\d{z}
      & \leq &
      \kappa L\left|t_{2}-t_{1}\right|,
      \\[10pt]
      \tv \left(p^{\kappa,\varepsilon}(t,\cdot), \mathcal{G}\right)
      & \leq &
      \Delta,
      &\int_{\mathcal{G}}
      \left|p^{\kappa,\varepsilon}(t_{2},z)-p^{\kappa,\varepsilon}(t_{1},z)\right|\d{z}
      & \leq &
      L\left|t_{2}-t_{1}\right|,
      \\[10pt]
      \tv \left(v^{\kappa,\varepsilon}(t,\cdot),\mathcal{G}\right)
      & \leq &
      \Delta,
      &
      \int_{\mathcal{G}}\left|v^{\kappa,\varepsilon}
        (t_{2},z)-v^{\kappa,\varepsilon}(t_{1},z)\right|\d{z}
      & \leq &
      L\left|t_{2}-t_{1}\right|,
      \\[10pt]
      \tv\left(\tau^{\kappa,\varepsilon}(t,\cdot),\mathcal{G}\right)
      & \leq &
      \Delta,
      &
      \int_{\mathcal{G}}\left|\tau^{\kappa,\varepsilon}
        (t_{2},z)-\tau^{\kappa,\varepsilon}(t_{1},z)\right|\d{z}
      & \leq &
      L\left|t_{2}-t_{1}\right|.
    \end{array}
  \end{equation}
  For any $z\in \mathcal{L},\;z_{1},z_{2}\in \mathcal{L} \setminus
  ([-\epsilon^2, \epsilon^2] \cup [m-\epsilon^2, m+\epsilon^2])$
  \begin{equation}
    \label{eq:spaceestimates}
    \begin{array}{@{}l@{\,}c@{\,}rl@{\,}c@{\,}r@{}}
      \tv\left(p^{\kappa,\varepsilon}(\cdot,z),\reali^{+}\right)
      & \leq &
      \frac{\Delta}{\kappa},
      &\int_{\reali^+}\left|p^{\kappa,\varepsilon}
        (t,z_{2})-p^{\kappa,\varepsilon}(t,z_{1})\right|\d{t}
      & \leq &
      L\left|z_{2}-z_{1}\right|,
      \\[10pt]
      \tv\left(v^{\kappa,\varepsilon}(\cdot,z),\reali^{+}\right)
      & \leq &
      \Delta,
      &\int_{\reali^+}\left|v^{\kappa,\varepsilon}
        (t,z_{2})-v^{\kappa,\varepsilon}(t,z_{1})\right|\d{t}
      & \leq &
      \kappa L\left|z_{2}-z_{1}\right|,
      \\[10pt]
      \tv\left(\tau^{\kappa,\varepsilon}(\cdot,z),\reali^{+}\right)
      & \leq &
      \kappa\Delta,
      &\int_{\reali^+}\left|\tau^{\kappa,\varepsilon}
        (t,z_{2})-\tau^{\kappa,\varepsilon}(t,z_{1})\right|\d{t}
      & \leq &
      \kappa^{2}L\left|z_{2}-z_{1}\right|.
    \end{array}
  \end{equation}
  For any $z,z_{1},z_{2}\in \mathcal{G}$
  \begin{equation}
    \label{eq:spaceestimates2}
    \begin{array}{@{}l@{\,}c@{\,}rl@{\,}c@{\,}r@{}}
      \tv\left(p^{\kappa,\varepsilon}(\cdot,z),\reali^{+}\right)
      & \leq &
      \Delta,
      &\int_{\reali^+}\left|p^{\kappa,\varepsilon}
        (t,z_{2})-p^{\kappa,\varepsilon}(t,z_{1})\right|\d{t}
      & \leq &
      L\left|z_{2}-z_{1}\right|,
      \\[10pt]
      \tv\left(v^{\kappa,\varepsilon}(\cdot,z),\reali^{+}\right)
      & \leq &
      \Delta,
      &\int_{\reali^+}\left|v^{\kappa,\varepsilon}
        (t,z_{2})-v^{\kappa,\varepsilon}(t,z_{1})\right|\d{t}
      & \leq &
      L\left|z_{2}-z_{1}\right|,
      \\[10pt]
      \tv\left(\tau^{\kappa,\varepsilon}(\cdot,z),\reali^{+}\right)
      & \leq &
      \Delta,
      &\int_{\reali^+}\left|\tau^{\kappa,\varepsilon}
        (t,z_{2})-\tau^{\kappa,\varepsilon}(t,z_{1})\right|\d{t}
      & \leq &
      L\left|z_{2}-z_{1}\right|.
    \end{array}
  \end{equation}
  For any $z,z_{1},z_{2}\in \reali$
  \begin{equation}
    \label{eq:spaceestimatesEpsilon}
    \begin{array}{@{}l@{\,}c@{\,}rl@{\,}c@{\,}r@{}}
      \tv\left(p^{\kappa,\varepsilon}(\cdot,z),\reali^{+}\right)
      & \leq &
      \frac{\Delta}{\kappa},
      &\int_{\reali^+}\left|p^{\kappa,\varepsilon}
        (t,z_{2})-p^{\kappa,\varepsilon}(t,z_{1})\right|\d{t}
      & \leq &
      \frac{L}{\kappa}\left|z_{2}-z_{1}\right|,
      \\[10pt]
      \tv\left(v^{\kappa,\varepsilon}(\cdot,z),\reali^{+}\right)
      & \leq &
      \Delta,
      &\int_{\reali^+}\left|v^{\kappa,\varepsilon}
        (t,z_{2})-v^{\kappa,\varepsilon}(t,z_{1})\right|\d{t}
      & \leq &
      L\left|z_{2}-z_{1}\right|.
    \end{array}
  \end{equation}
  Moreover, the maximal size of rarefaction waves is uniformly bounded
  by a constant, independent of $\kappa$, times $\epsilon$.
\end{proposition}

\begin{proof}
  Choose $\bar\delta$ as in Lemma~\ref{lem:Upsilon} and $\kappa_*$ as
  in Lemma~\ref{lem:InterInter}. Define $\delta = \bar\delta/C$, with
  $C$ as in Lemma~\ref{lem:wtw}. Using the piecewise constant initial
  data $\tilde u$, we use the previously described algorithm and call
  $u^{\kappa,\varepsilon}$ the piecewise constant approximate solution
  so obtained. By Lemma~\ref{lem:wtw}, we have
  $\Upsilon\left(u^{\kappa,\varepsilon}(0+)\right)\leq C\cdot
  \wtv_{\kappa}\left(\tilde u\right)<\bar \delta$. By
  Lemma~\ref{lem:Upsilon}, the map
  $t\to\Upsilon\left(u^{\kappa,\varepsilon}(t)\right)$ is not
  increasing so that $\Upsilon\left(u^{\kappa,\varepsilon}(t)\right)<
  \bar \delta$ for all positive times. Lemma~\ref{lem:nocluster}
  ensures that $u^{\kappa,\varepsilon}$ can be constructed for all
  times $t\geq 0$.  Again, Lemma~\ref{lem:wtw} implies the estimate
  \begin{displaymath} {\wtv}_\kappa \left(u^{\kappa,\epsilon}
      (t)\right) \leq C \, \Upsilon\left(u^{\kappa,\epsilon}
      (t)\right) \leq C \, \Upsilon\left(u^{\kappa,\epsilon}(0+)
    \right)\leq C^{2} \, {\wtv}_{\kappa}\left(\tilde u\right) \leq
    C^{2} \, \delta \, =C\bar\delta.
  \end{displaymath}
  The estimates on the total variation of $p^{\kappa,\epsilon}$ and
  $v^{\kappa,\epsilon}$ in~\eqref{eq:timeestimates} immediately
  follow. To obtain the bounds on the total variation of the specific
  volume in the liquid, use~\eqref{eq:tkfamily}. The Lipschitz
  continuity estimates in~\eqref{eq:timeestimates} are now a standard
  consequence, see e.g.~\cite[Section~7.4]{BressanLectureNotes}, since
  the wave propagation speed in the gas is uniformly bounded
  independently of $\kappa$ and in the liquid (also in
  $\mathcal{I}_\epsilon^{\pm}$) is bounded by $\O/\kappa$.

  \smallskip

  Pass to~\eqref{eq:spaceestimates}.  Observe first that from the
  proof of Lemma~\ref{lem:Upsilon} it follows that
  \begin{equation}
    \label{eq:3}
    6 \, C^{3} \bar  \delta \leq 1
    \qquad \mbox{ and } \qquad
    C\geq 1 \,.
  \end{equation}
  As usual, we call $\sigma_\alpha$ the size of the wave supported at
  $z_\alpha$. For $z \in \mathcal{L}$, define
  \begin{eqnarray*}
    \Upsilon_z (t)
    & = &
    W_z (t) + V_z (t) + \frac{3C}{\kappa} \; \Upsilon (t)
    \\
    W_z (t)
    & = &
    \sum_{\tau \in [0, t]} \modulo{\Delta p (\tau, z)}
    \\
    V_z (t)
    & = &
    \sum_{\alpha \in I_z (t)} \modulo{\sigma_\alpha}
    \\
    I_z (t)
    & = &
    \left\{
      \alpha \colon
      z_{\alpha}\in\mathcal{L}\mbox{ and the wave at }z_{\alpha} \mbox{ is of the }
      \left\{
        \begin{array}{l}
          \mbox{first family and } z_\alpha > z
          \\
          \mbox{second family and } z_\alpha < z
        \end{array}
      \right.
    \right\}\\
    \Upsilon(t)&=&\Upsilon\left(u^{\kappa,\varepsilon}(t)\right)
  \end{eqnarray*}
  Note that the sum defining $W_z$ is actually a finite sum, since the
  total number of waves is finite by Lemma~\ref{lem:nocluster}. We
  claim that $t \to \Upsilon_z (t)$ is non increasing. Indeed,
  $\Upsilon_z (t)$ may change its value at a time $t$ when:
  \begin{enumerate}
  \item A wave with size $\sigma_{\bar\alpha}$ crosses $z$ and no
    other interaction occurs. Then, $\Delta W_z (t) = \modulo{\Delta p
      (t, z)} = \modulo{\sigma_{\bar\alpha}}$, $\Delta V_z (t) =
    -\modulo{\sigma_{\bar\alpha}}$ and $\Delta\Upsilon (t) =
    0$. Hence, $\Delta \Upsilon_z (t) = 0$.
  \item An interaction in $\mathcal{G}$ occurs and no wave crosses
    $z$. Then, $\Delta W_z (t) =0$, $\Delta V_z (t) = 0$ and $\Delta
    \Upsilon (t) \leq 0$. Hence, $\Delta \Upsilon_z (t) \leq 0$.
  \item An interaction in $\mathcal{L}$ occurs and no wave crosses
    $z$. Then, $\Delta{ W_z (t)} = 0$; calling $\sigma_\alpha$,
    $\sigma_\beta$ the sizes of the interacting waves, $\Delta V_z (t)
    \leq C \kappa^2 \modulo{\sigma_\alpha \, \sigma_\beta}$
    by~\eqref{eq:ieNonStandard} and $\Delta \Upsilon (t) \leq
    -\kappa^2 \modulo{\sigma_\alpha \sigma_\beta}$
    by~\eqref{eq:1}. Hence, $\Delta\Upsilon_z (t) \leq C\left(1 -
      \frac{3}{\kappa}\right) \kappa^2 \modulo{\sigma_\alpha
      \sigma_\beta} \leq 0$.
  \item A $2$-wave with size $\sigma_\alpha$, coming from
    $\mathcal{G}$, and a $1$-wave with size $\sigma_\beta$, coming
    from $\mathcal{L}$, interact at $\bar z=0$. Then, $\Delta W_z (t)
    = 0$; by Lemma~\ref{lem:InterInter}, $\Delta V_z (t) \leq
    (1-c\kappa) \modulo{\sigma_\beta} + (2+C\bar\delta)
    \modulo{\sigma_\alpha}$; by Lemma~\ref{lem:Upsilon} $\Delta
    \Upsilon (t) \leq -\modulo{\sigma_\alpha } - \kappa
    \modulo{\sigma_\beta}$. Hence, $\Delta\Upsilon_z (t) \leq
    \left(1-c\kappa-3C\right) \modulo{\sigma_\beta} +
    \left(2+C\bar\delta-\frac{3C}{\kappa}\right)
    \modulo{\sigma_{\alpha}} \leq 0$.
  \item Two waves interact at $z=m$: the same procedure as above
    applies.
  \end{enumerate}
  The remaining times where $\Upsilon_z$ may change value consist in
  the superposition of two or more of the cases considered above and
  can be dealt superimposing the corresponding
  inequalities. Therefore,
  \begin{displaymath}
    \begin{array}{rclclcl}
      \tv\left(p^{\kappa,\epsilon} (\cdot, z)\right)
      & = &
      \displaystyle
      \sup_{T>0} \tv\left(p^{\kappa,\epsilon} (\cdot, z); [0,T]\right)
      & =&
      \displaystyle
      \sup_{T>0} W_z (T)
      & \leq &
      \displaystyle
      \sup_{T>0} \Upsilon_z (T)
      \\[12pt]
      & \leq &
      \Upsilon_z (0+)
      & = &
      V_z (0+) + \frac{3C}{\kappa} \Upsilon (0+)
      & \leq &
      \frac{\Delta}{\kappa}
    \end{array}
  \end{displaymath}
  provided $\Delta > 2 \bar \delta$, completing the proof of the first
  estimate on the total variation in~\eqref{eq:spaceestimates}. The
  remaining total variation bounds in~\eqref{eq:spaceestimates} follow
  from the estimates
  \begin{displaymath}
    \modulo{\Delta v (t,z)} \leq \O \kappa \; \modulo{\Delta p (t,z)}
    \quad \mbox{ and } \quad
    \modulo{\Delta \tau (t,z)} \leq \O \kappa^2 \; \modulo{\Delta p (t,z)}
  \end{displaymath}
  which hold along Lax curves by Lemma~\ref{lem:para}
  and~\eqref{eq:tkfamily}.  The Lipschitz continuity estimates
  in~\eqref{eq:spaceestimates} are now a standard consequence, see
  e.g.~\cite[Section~7.4]{BressanLectureNotes}, since the wave
  propagation speed in $\mathcal{L} \setminus ([-\epsilon^2,
  \epsilon^2] \cup [m-\epsilon^2, m+\epsilon^2])$ is of order
  $1/\kappa$.

  \smallskip

  The proof of the estimates~\eqref{eq:spaceestimates2} is obtained
  from that of~\eqref{eq:spaceestimates} completed above, formally
  setting $\kappa = 1$ and with obvious modifications to the
  definition of $\Upsilon_z$.

  The estimates on all the real line~\eqref{eq:spaceestimatesEpsilon}
  are obtained choosing a common upper bound on the total variation
  and a common lower bound on the wave speeds in the liquid, in the
  gas and in the two strips $\mathcal{I}^{\pm}_{\varepsilon}$ and
  observing that $z\mapsto u^{\kappa,\varepsilon}\left(t,z\right)$ is
  continuous at $z=0$, $z=m$ for every $t\geq 0$ in which no wave
  interacts with the interfaces. Observe that a similar Lipschitz
  estimate does not hold for the specific volume $\tau$, since at
  $z=0$ and $z=m$ it is not continuous.  \smallskip

  Finally, the estimate on the maximal size of rarefaction waves
  follows the lines in~\cite[Section~7.3,
  Step~5]{BressanLectureNotes}. Indeed, call $\bar\sigma(t)$ the size
  at time $t$ of a rarefaction wave in the wave front tracking
  approximation. We claim that, if in the interval $[t_{o},\tau]$ the
  wave does not leave the phase in which it is found at time $t_{o}$
  and does not disappear due to possible interactions with shocks of
  the same family, then $\left|\bar\sigma(\tau)\right|\leq 6
  \left|\bar \sigma(t_{o})\right|$.

  Indeed, consider the liquid phase, let $\bar z(t)$ be the location
  of the wave at time $t$ and define
  \begin{displaymath}
    \begin{array}{rcl}
      s(t) & = &
      \left|\bar \sigma(t)\right|
      \left[1+6 \, C^{2} \, \kappa^{2} \, V_{s}(t)
        +
        24 \, C^{3} \, \kappa \, \Upsilon(t)
      \right]
      \\
      V_s (t)
      & = &
      \sum_{\alpha \in I_{s} (t)} \modulo{\sigma_\alpha}
      \\
      I_{s} (t)
      & = &
      \left\{
        \alpha \colon
        z_{\alpha}\in\mathcal{L} \mbox{, the wave at }z_{\alpha}
        \mbox{ is approaching the wave at  } \bar z
      \right\}.
    \end{array}
  \end{displaymath}
  The function $t \to s(t)$ is non increasing in the interval
  $\left[t_{o},\tau\right]$.  Indeed, $s(t)$ may change its value at
  the following times:
  \begin{enumerate}
  \item At time $t$ a wave $\sigma_{\alpha}$ interacts with the wave
    at $\bar z(t)$ and no other interaction occurs. Then,
    by~\eqref{eq:ieNonStandard} $\Delta \left|\bar
      \sigma(t)\right|\leq C \kappa^{2}\left|\bar
      \sigma(t-)\sigma_{\alpha}\right|$; $\Delta V_{s} (t) =
    -\left|\sigma_{\alpha}\right| $; $\Delta\Upsilon(t)< 0$.  Hence,
    by~\eqref{eq:3},
    \begin{displaymath}
      \begin{split}
        \!\!\!\!  \Delta s(t) &= \Delta \left|\bar
          \sigma(t)\right|\left[1+6C^{2}\kappa^{2}V_{s}(t+)+
          24C^{3}\kappa\Upsilon(t+) \right]+
        \left|\bar\sigma(t-)\right|\left[6C^{2}\kappa^{2}\Delta
          V_{s}(t)+ 24C^{3}\kappa\Delta\Upsilon(t)\right]
        \\
        &\leq C\kappa^{2}\left|\bar \sigma(t-)\sigma_{\alpha}\right|
        \left[1+ 6C^{2}\kappa^{2}\bar \delta+24C^{3}\kappa\bar
          \delta\right] -
        6C^{2}\kappa^{2}\left|\bar\sigma(t-)\right|\left|\sigma_{\alpha}\right|
        \\
        &\leq C\kappa^{2}\left|\bar \sigma(t-)\sigma_{\alpha}\right|
        \left[1+ 6C^{2}\kappa^{2}\bar \delta+24C^{3}\kappa\bar \delta
          - 6C\right]
        \\
        &\leq C\kappa^{2}\left|\bar \sigma(t-)\sigma_{\alpha}\right|
        \left[1+ 1+4 - 6C\right]\leq 0 \,.
      \end{split}
    \end{displaymath}
  \item At time $t$, an interaction in $\mathcal{G}$ occurs and no
    wave crosses $\bar z(t)$. Then, $\Delta \left|\bar\sigma
      (t)\right| =0$, $\Delta V_{s} (t) = 0$ and $\Delta \Upsilon (t)
    \leq 0$. Hence, $\Delta s (t) \leq 0$.
  \item At time $t$, an interaction in $\mathcal{L}$ occurs and no
    wave crosses $\bar z(t)$. Then, $\Delta\left|\bar\sigma (t)\right|
    = 0$; calling $\sigma_\alpha$, $\sigma_\beta$ the sizes of the
    interacting waves, $\Delta V_{s} (t) \leq C\kappa^2
    \modulo{\sigma_\alpha \, \sigma_\beta}$
    by~\eqref{eq:ieNonStandard} and $\Delta \Upsilon (t) \leq
    -\kappa^2 \modulo{\sigma_\alpha \sigma_\beta}$
    by~\eqref{eq:1}. Hence,
    \begin{displaymath}
      \Delta s (t) \leq
      \left|\bar\sigma(t)\right|
      \left[6C^{3}\kappa^{4}\left|\sigma_{\alpha}\sigma_{\beta}\right|
        -24C^{3}\kappa^{3}\left|\sigma_{\alpha}\sigma_{\beta}\right|\right]
      \leq
      \left|\bar\sigma(t)\right|
      \left|\sigma_{\alpha}\sigma_{\beta}\right|
      6C^{3}\kappa^{3}(\kappa -4) \leq 0 \,.
    \end{displaymath}
  \item At time $t$, an interaction occurs at $z=0$ and no wave
    crosses $\bar z(t)$. Call $\sigma_\alpha$ the size of the wave
    coming from $\mathcal{G}$, and $\sigma_\beta$ the size of the wave
    coming from $\mathcal{L}$. Then, $\Delta \left|\bar
      \sigma(t)\right| = 0$. By Lemma~\ref{lem:InterInter}, $\Delta
    V_s (t) \leq (1-c\kappa) \modulo{\sigma_\beta} + (2+C\bar\delta)
    \modulo{\sigma_\alpha}$; by Lemma~\ref{lem:Upsilon} $\Delta
    \Upsilon (t) \leq -\modulo{\sigma_\alpha } - \kappa
    \modulo{\sigma_\beta}$. Hence,
    \begin{displaymath}
      \begin{split}
        \Delta s(t) &\leq \left|\bar \sigma(t)\right|\left[
          6C^{2}\kappa^{2}\left((1-c\kappa) \modulo{\sigma_\beta} +
            (2+C\bar\delta) \modulo{\sigma_\alpha}\right)
          -24C^{3}\kappa\left(\modulo{\sigma_\alpha } + \kappa
            \modulo{\sigma_\beta}\right) \right]\\
        &\leq 6C^{2}\kappa\left|\bar \sigma(t)\right|\left[
          \kappa\left((1-c\kappa) \modulo{\sigma_\beta} +
            (2+C\bar\delta) \modulo{\sigma_\alpha}\right)
          -4C\left(\modulo{\sigma_\alpha } + \kappa
            \modulo{\sigma_\beta}\right) \right]\\
        &\leq 6C^{2}\kappa\left|\bar \sigma(t)\right|\left[
          \kappa\left|\sigma_{\beta}\right|\left(1-c\kappa-4C\right)
          +\left|\sigma_{\alpha}\right|\left(2\kappa+C\kappa\bar
            \delta-4C\right)\right]\leq 0
      \end{split}
    \end{displaymath}
  \item Two waves interact at $z=m$: the same procedure as above
    applies.
  \end{enumerate}
  The remaining times where $s(t)$ may change value consist in the
  superposition of two or more of the cases considered above and can
  be dealt superimposing the corresponding inequalities proved
  above. Therefore $s(\tau)\leq s(t_{o})$ which implies
  \begin{displaymath}
    \left|\bar \sigma(\tau)\right|
    \leq
    \left|\bar \sigma(t_{o})\right|
    \frac{1+6 C^{2} \kappa^{2}V_{s}(t_{o})+24C^{3}\kappa\Upsilon(t_{o})}
    {1+6 C^{2} \kappa^{2}V_{s}(\tau)+24C^{3}\kappa\Upsilon(\tau)}
    \leq
    \left[1+6 C^{2} \bar\delta+24C^{3}\bar\delta\right]
    \left|\bar \sigma(t_{o})\right|
    \leq
    6\left|\bar \sigma(t_{o})\right|
  \end{displaymath}
  This proves the claim in the liquid. In the case of a wave in the
  gas, the argument is similar: it is sufficient to set $\kappa=1$ in
  the definition of $s(t)$ and make the obvious modifications to the
  map $V_s$.

  Finally, we observe now that when a wave crosses the interfaces, the
  refracted wave has a strength given by the strength of the incoming
  wave times a constant bounded uniformly with respect to $\kappa$,
  for instance we can choose $3C$ (see Lemma~\ref{lem:InterInter}).
  Moreover, when a rarefaction is born, its strength is less than
  $\epsilon$ and it can cross at most an interface once. Therefore,
  also the last claim of the Proposition is proved with the constant
  $1944 \, C^{2}\varepsilon$.
\end{proof}

\begin{proofof}{Theorem~\ref{thm:kappa}}
  Use $\delta, \Delta, L, \kappa_* > 0$ as defined in
  Proposition~\ref{prop:approxsol} and choose any $\kappa\in
  \left]0,\kappa^{*} \right[$. Fix a suitable sequence $\epsilon_\nu$
  strictly decreasing to $0$. Approximate the initial datum $\tilde
  u=\left(\tilde p,\tilde v\right)$ with an approximate, piecewise
  constant initial datum $\tilde u^{\epsilon_\nu}$
  satisfying~\eqref{eq:hyp_id}, so that $\wtv_{\kappa}\left(\tilde
    u^{\varepsilon_{\nu}}\right)\le \wtv_{\kappa}\left(\tilde
    u\right)\le \delta$, $\|\tilde
  p^{\varepsilon_{\nu}}-p_{o}\|_{\L\infty}\le \delta$.

  Proposition~\ref{prop:approxsol} ensures that it is possible to
  construct a wave front tracking $\epsilon_\nu$--approximate solution
  $(p^{\kappa,\epsilon_\nu}, v^{\kappa,\epsilon_\nu})$ that satisfies
  all properties stated therein.

  Using~\eqref{eq:timeestimates} and~\eqref{eq:spaceestimatesEpsilon},
  a repeated application of Helly
  Theorem~\cite[Theorem~2.4]{BressanLectureNotes}, ensures the
  convergence of a suitable subsequence, which we still denote by
  $u^{\kappa,\epsilon_\nu}$, to a function $u^{\kappa} =
  \left(p^{\kappa},v^{\kappa}\right)$ in the following sense
  \begin{eqnarray*}
    \lim_{\nu \to +\infty}
    \norma{(p^{\kappa,\epsilon_\nu}, v^{\kappa,\epsilon_\nu}) (t, \cdot)
      -
      (p^\kappa, v^\kappa) (t, \cdot)}_{\L1 ([-M,M]; \reali^+ \times \reali)}
    & = & 0, \mbox{ for any }t\geq 0,\;M>0
    \\
    \lim_{\nu \to +\infty}
    \norma{(p^{\kappa,\epsilon_\nu}, v^{\kappa,\epsilon_\nu}) (\cdot,z)
      -
      (p^\kappa, v^\kappa) (\cdot,z)}_{\L1 ([0,M]; \reali^+ \times \reali)}
    & = & 0, \mbox{ for any }z\in\reali,\; M>0
    \\
    \left(p^{\kappa},v^{\kappa}\right)(0,\cdot)
    & = &
    \left(\tilde p^{\kappa},\tilde v^{\kappa}\right)(\cdot).
  \end{eqnarray*}
  Passing to the limit in~\eqref{eq:timeestimates},
  \eqref{eq:spaceestimates}, \eqref{eq:spaceestimates2},
  \eqref{eq:spaceestimatesEpsilon}, we
  obtain~\eqref{eq:THMtimeestimates}, \eqref{eq:THMspaceestimates},
  \eqref{eq:THMspaceestimates2} and~\eqref{eq:spaceestimatesAllLine}.

  Since the bounds on the total variation are uniform in $\varepsilon$
  and since the strength of rarefactions is uniformly bounded by a
  constant times $\varepsilon$, standard techniques in wave front
  tracking~\cite[Section 7.4]{BressanLectureNotes} can be used to show
  that the limit $u^{\kappa}$ is a weak entropy solution
  to~\eqref{eq:finalequation} in the open regions $z<0$, $0<z<m$,
  $z>m$. By~\eqref{eq:spaceestimatesAllLine}, we have that the map $z
  \to u^{\kappa} (\cdot, z)$ is continuous in $\L1$, in particular it
  is continuous across $z=0$ and $z=m$. Therefore, $u^\kappa$
  trivially satisfies there the Rankine-Hugoniot conditions and the
  entropy (in)equality. Hence, $u^{\kappa}$ is a weak entropy solution
  to~\eqref{eq:finalequation} in all $ \reali^+ \times \reali$.
\end{proofof}

\begin{proofof}{Theorem~\ref{thm:limit}}
  By~\eqref{eq:questa}, $\wtv_\kappa (\tilde u) < \delta$ so that
  Theorem~\ref{thm:kappa} applies, ensuring the existence of a
  solution $u^\kappa = (p^\kappa, v^\kappa)$
  to~\eqref{eq:finalequation} satisfying~\eqref{eq:THMtimeestimates},
  \eqref{eq:THMspaceestimates}, \eqref{eq:THMspaceestimates2}
  and~\eqref{eq:spaceestimatesAllLine}.

  Since $\kappa<1$, from~\eqref{eq:THMtimeestimates}
  and~\eqref{eq:spaceestimatesAllLine} we have for $v^{\kappa}$:
  \begin{equation}
    \label{eq:lipallreal}
    \begin{array}{@{}rclrcll@{}}
      \tv\left(v^{\kappa}(t,\cdot),\reali\right)
      & \leq &
      \Delta,
      &
      \int_{\reali}\left|v^{\kappa}
        (t_{2},z)-v^{\kappa}(t_{1},z)\right| \d{z}
      & \leq &
      L\left|t_{2}-t_{1}\right|,
      & t,t_{1},t_{2}\geq  0,
      \\
      \tv\left(v^{\kappa}(\cdot,z),\reali^{+}\right) & \leq & \Delta,
      &
      \int_{\reali^{+}}\left|v^{\kappa} (t,z_{2})-v^{\kappa}(t,z_{1})\right| \d{t}
      & \leq &
      L\left|z_{2}-z_{1}\right|,
      & z,z_{1},z_{2}\in\reali.
    \end{array}
  \end{equation}
  Helly Theorem~\cite[Theorem~2.4]{BressanLectureNotes} implies the
  existence of a subsequence (that we call again $v^{\kappa}$)
  converging to a limit $v^{*}$ in the sense
  of~\eqref{eq:vconvergence}.  From the bound
  in~\eqref{eq:THMtimeestimates} on the total variation of
  $v^{\kappa}$ or from the Lipschitz estimate
  in~\eqref{eq:THMspaceestimates} for $v^{\kappa}$ in the liquid, it
  is straightforward to obtain that $v^{*}(t,z)=v_{l}(t)$ for all
  $z\in\mathcal{L}$ and $t\geq 0$, where $v_{l}(t)$ is a function
  which depends on time only, completing the proof
  of~\eqref{eq:vconvergence} and of 3.~in Definition~\ref{def:sol2}.

  The same procedure can be carried out for the pressure in the gas
  region, proving the first four lines in~\eqref{eq:pconvergence}.
  Observe that for the pressure, we cannot apply Helly Theorem in the
  liquid since there the estimates blow up as $\kappa\to 0$.  Because
  of the strong convergence in the gas region of both the velocity and
  the pressure, the limit $u^{*}=\left(p^{*},v^{*}\right)$ satisfies
  1.~in Definition~\ref{def:sol2} and the initial condition
  $u^{*}(0,z)=\tilde u(z)$ a.e.~$z\in\mathcal{G}$.

  The uniform convergence of $\tau^{\kappa}$ in the liquid is a
  straightforward consequence of~\eqref{eq:tkfamily} and of the
  uniform bound on the $\L\infty$ norm of $p^{\kappa}$.

  Since the pressure is uniformly bounded, we have a weak$^\star$
  convergence (possibly passing to further subsequences) $p^{\kappa}
  \wsto p^{*}$ in
  $\L\infty\left(\reali^{+}\times\reali,\reali\right)$~\cite[Section~4.3
  Point~C.]{Brezis}. If we define
  $p_{l}=p^{*}_{|{\reali^{+}\times\mathcal{L}}}$ we get the fifth line
  in~\eqref{eq:pconvergence}.

  By~\eqref{eq:THMtimeestimates} and~\eqref{eq:spaceestimatesAllLine},
  the second equation in~\eqref{eq:finalequation} can be written in
  integral form in $\left[t_{1},t_{2}\right]\times\mathcal{L}$:
  \begin{equation}
    \label{eq:secondintegraleq}
    \int_{0}^{m}v^{\kappa}\left(t_{1},z\right)\;\d{z}
    -\int_{0}^{m}v^{\kappa}\left(t_{2},z\right)\;\d{z}
    +\int_{t_{1}}^{t_{2}}p^{\kappa}\left(t,0\right)\;\d{t}
    -\int_{t_{1}}^{t_{2}}p^{\kappa}\left(t,m\right)\;\d{t}=0.
  \end{equation}
  Now, we use the strong convergence of both $p^{\kappa}$ and
  $v^{\kappa}$ in the gas region and the fact that in $\mathcal{L}$,
  $v^{*}$ is constant to obtain
  \begin{displaymath}
    m\left[v_{l}\left(t_{2}\right)-v_{l}\left(t_{1}\right)\right]
    =
    \int_{t_{1}}^{t_{2}}p^{*}\left(t,0\right)\;\d{t}
    -\int_{t_{1}}^{t_{2}}p^{*}\left(t,m\right)\;\d{t}.
  \end{displaymath}
  Setting $t_{1}=0$ and $t_{2}=t$ in the last expression above,
  \begin{displaymath}
    v_{l}\left(t\right)=
    v_{l}\left(0\right)
    +\frac{1}{m}
    \int_{0}^{t}
    \left[p^{*}\left(s,0\right) - p^{*}\left(s,m\right)\right]\; \d{s}
    =
    v_{l}\left(0\right) + \frac{1}{m}
    \int_{0}^{t}\left[p^{*}\left(s,0-\right)
      -p^{*}\left(s,m+\right)\right]\;\d{s}
  \end{displaymath}
  which means that $v_{l}$ is Lipschitz continuous and satisfies 2.~in
  Definition~\ref{def:sol2}.

  Observe that the non linear term
  $\mathcal{T}_{\kappa}\left(z,p^{\kappa}\right)$ converges strongly
  to
  \begin{displaymath}
    \tau^{*}(t,z)=
    \begin{cases}
      \bar \tau & \mbox{ for }z\in\mathcal{L},
      \\
      \mathcal{T}_{g}\left(p^{*}(t,z)\right) &\mbox{ for } z \in
      \mathcal{G},
    \end{cases}
  \end{displaymath}
  hence we can pass to the limit in~\eqref{eq:finalequation} in
  distributional sense to obtain
  \begin{equation}
    \label{eq:finaleqlimit}
    \begin{cases}
      \partial_{t}\tau^{*}-\partial_{z}v^{*}=0
      \\
      \partial_{t}v^{*}+\partial_{z}p^{*}=0,
    \end{cases}
    \mbox{ in }\reali^{+}\times\reali.
  \end{equation}
  Since in the liquid region $v^{*}\left(t,z\right)=v_{l}(t)$ with
  $v_{l}$ Lipschitz continuous, the second equation
  in~\eqref{eq:finaleqlimit} becomes
  \begin{displaymath}
    \partial_{z} p^* (t,z)= - \dot v_{l}(t)
    \mbox{ in }\reali^{+}\times\mathcal{L}.
  \end{displaymath}
  Therefore there exists a measurable function $\beta(t)$ such that
  the function
  \begin{displaymath}
    p_{l}(t,z) = - z \dot v_{l}(t) + \beta (t)
  \end{displaymath}
  can be chosen as a representative of the limit pressure $p^{*}$
  restricted to the liquid. This implies the existence of the two
  limits
  \begin{displaymath}
    \lim_{z\to 0^{+}}p_{l}(t,z) = \beta(t),
    \qquad
    \lim_{z\to m^{-}} p_{l}(t,z) = \beta(t)-z\dot v_{l}(t).
  \end{displaymath}
  The fourth line in~\eqref{eq:pconvergence} ensures the existence of
  the corresponding limits from the gas region:
  \begin{displaymath}
    \lim_{z\to 0^{-}} p^{*}(t,z)=p^{*}(t,0),
    \qquad
    \lim_{z\to m^{+}} p^{*}(t,z)=p^{*}\left(t,m\right),
    \mbox{ a.e. }t\in\reali^{+},
  \end{displaymath}
  hence Rankine-Hugoniot conditions for~\eqref{eq:finaleqlimit}
  applied along $z=0$ and $z=m$ imply that the right and the left
  limit of the pressure must coincide along $z=0$ and $z=m$ for
  a.e.~$t\geq 0$. Therefore, we have
  \begin{displaymath}
    \begin{cases}
      p^{*}(t,0)=\beta(t)
      \\
      p^{*}(t,m)=-m\dot v_{l}(t)+\beta(t)
    \end{cases}
    \mbox{ for a.e. }t\geq 0,
  \end{displaymath}
  which implies the remaining equality to be proved
  in~\eqref{eq:pconvergence}.
\end{proofof}

\noindent\textbf{Acknowledgment:} The present work was supported by
the PRIN~2012 project \emph{Nonlinear Hyperbolic Partial Differential
  Equations, Dispersive and Transport Equations: Theoretical and
  Applicative Aspects} and by the Gruppo Nazionale per l'Analisi
Matematica,
la Probabilità e le loro Applicazioni (GNAMPA) of the Istituto Nazionale di Alta Matematica (INdAM).

{\small

  \bibliographystyle{abbrv}

  \bibliography{IncompressibleLimit}

}

\end{document}